\documentclass{cocv}
\usepackage{graphicx, amssymb, amsmath, epsfig}
\usepackage{color}
\usepackage{soul} 
\usepackage{amsmath}%
\usepackage{booktabs}
\usepackage{float} 
\usepackage{boxedminipage}
\usepackage{amsfonts}
\usepackage{algorithmic}
\usepackage{algorithm,booktabs}
\usepackage[mathscr]{eucal}
\usepackage{tikz}
\usepackage{todonotes}
\usepackage{caption} 
\tikzset{
  treenode/.style = {shape=rectangle, rounded corners,
                     draw, align=center,
                     top color=white, bottom color=blue!20},
  root/.style     = {treenode, font=\Large, bottom color=red!30},
  env/.style      = {treenode, font=\ttfamily\normalsize},
  dummy/.style    = {circle,draw}
}

\newtheorem{lemma}{Lemma}[section]
\newtheorem{theorem}{Theorem}[section]
\newtheorem{rmk}{Remark}[section]
\newtheorem{proposition}{Proposition}[section]

\usepackage{mathtools}
\DeclarePairedDelimiter\floor{\lfloor}{\rfloor}

\newcommand{\ep}{\varepsilon_{\mathcal{T}}}
\newcommand{\R}{\mathbb{R}}
\DeclareMathOperator*{\argmax}{arg\,max}
\DeclareMathOperator*{\argmin}{arg\,min}
\DeclareMathOperator*{\supp}{supp}

\begin{document}
%


\title{Error estimates for a tree structure algorithm solving finite horizon control problems}

\author{Luca Saluzzi$^1$}
\author{Alessandro Alla$^2$}
\author{Maurizio Falcone}

\address{Department of Mathematics, Imperial College London, South Kensington Campus, SW7 2AZ London, United Kingdom, l.saluzzi@imperial.ac.uk}
\address{Dipartimento di Scienze Molecolari e Nanosistemi, Università Ca' Foscari Venezia, Itay, alessandro.alla@unive.it}
\address{Sapienza Universit\`a  di Roma - Piazzale Aldo Moro 5, 00185 Roma, falcone@mat.uniroma1.it}

\begin{abstract}
In the Dynamic Programming approach to optimal control problems a crucial role is played by the value function that is characterized as the unique viscosity solution of a Hamilton-Jacobi-Bellman (HJB) equation. It is well known that this approach suffers from the "curse of dimensionality" and this limitation has reduced its use in real world applications. Here, we analyze a dynamic programming algorithm based on a tree structure to mitigate the "curse of dimensionality". The tree is built by the discrete time dynamics avoiding the use of a fixed space grid which is the bottleneck for high-dimensional problems, this also drops the projection on the grid in the approximation of the value function.  In this work, we present first order error estimates for the the approximation of the value function based on the tree-structure algorithm. The estimate turns out to have the same order of convergence of the numerical method used for the approximation of the dynamics. Furthermore, we analyze a pruning technique for the tree to reduce the complexity and minimize the computational effort. Finally, we present some numerical tests to show the theoretical results.

\end{abstract}
\begin{resume} 
Dans l'approche de programmation dynamique des probl\`emes de contr\^ole optimal, un r\^ole crucial est jou\'e par la fonction de valeur qui est caract\'eris\'e e comme la solution de viscosit\'e unique d'une \'equation de Hamilton-Jacobi-Bellman (HJB). Il est bien connu que cette approche souffre de la mal\'ediction de la dimensionnalit\'e et cette limitation a r\'eduit son utilisation dans les applications du monde r\'eel. Ici, nous analysons un algorithme de programmation dynamique bas\'e sur une structure arborescente. L'arbre est construit par la dynamique discr\`ete temporelle \'evitant l'utilisation d'une grille spatiale fixe qui est le goulot d'\'etranglement pour les probl\`emes de grande dimension, cela fait \'egalement tomber la projection sur la grille dans l'approximation de la fonction valeur.
  
Dans ce travail, nous pr\'esentons des estimations d'erreur de premier ordre pour l'approximation de la fonction de valeur bas\'ee sur l'algorithme de structure arborescente. L'estimation s'av\`ere avoir le m\^{e}me ordre de convergence que la m\'ethode num\'erique utilis\'ee pour l'approximation de la dynamique.
De plus, nous analysons une technique d'\'elagage de l'arbre pour r\'eduire la complexit\'e et minimiser l'effort de calcul. Enfin, nous pr\'esentons quelques tests num\'eriques pour montrer les r\'esultats th\'eoriques.
\end{resume}


\maketitle



\section{Introduction}
The Dynamic Programming (DP) approach introduced by Bellman (see e.g. \cite{B57}) has been applied to several deterministic and stochastic optimal control problems in finite dimension. This approach has been revitalized thanks to the theory of weak solutions for Hamilton-Jacobi equations, the so-called viscosity solutions,  introduced by Crandall and Lions in the middle of the 80s (see the monographs \cite{BCD97} and \cite{FS93} and list of references therein). Despite the huge amount of theoretical results and the numerical methods devoted to develop efficient and accurate  algorithms for Hamilton-Jacobi equations, real applications of DP has been up to now limited to rather low dimensional problems.  The solution of many optimal control problems (and in particular those governed by evolutive partial differential equations) is still  accomplished  via  open-loop controls, see e.g. \cite{HPUU09}. In fact DP provides an elegant characterization of the value function as the unique viscosity solution of a nonlinear partial differential equation (the Hamilton-Jacobi-Bellman equation) which is usually computed on a space grid, this is a major bottleneck for high-dimensional problems. However, this remains an interesting and challenging problem since by an approximate knowledge of the value function one can derive a synthesis of a feedback control law that can be plugged into the controlled dynamics. This remarkable feature of DP allows for a synthesis that can be applied  to control problems with non linear dynamics and running costs. The case of a linear dynamics and quadratic costs (LQR problem) has an explicit solution based on the Riccati equation (we refer to the book \cite{BIM12} for a general introduction to numerical methods for the Riccati equation). It is interesting to note that this equation can be solved even in very high-dimensional spaces as in \cite{SSM14, S16} using Krylov subspaces. We also refer to the recent paper \cite{BBKS18} for  a comparison of various techniques.

As we said,  DP suffers from the {\em curse of dimensionality} and even in low dimension an accurate approximation of viscosity solutions  is a challenging problem due to their lack of regularity (the value function is in general just Lipschitz continuous even for regular dynamics and costs).  However, the analysis of low order numerical methods is now rather complete even for a state space in $\R^d$ and several methods to solve HJB equations are available (see the monographies by Sethian \cite{S99}, Osher and Fedkiw \cite{OF03}, Falcone and Ferretti \cite{FF13} for an extensive discussion of some of these methods). From the practical point of view all the classical  PDE methods require a space discretization based on a space grid (or triangulation) and this implies a huge amount of memory allocations for high dimensional problems and makes the problem unfeasible for a dimension $d>5$ on a standard computer. For some of these methods, a-priori error estimates are available, in particular the construction of a DP algorithm for  time dependent problems has been addressed in \cite{FG99}.

Several efforts have been made to mitigate the {\em curse of dimensionality} for the numerical solutions of deterministic optimal control problems.  Although a detailed description of these contributions goes beyond the scope of this paper, let us briefly mention for the sake of completeness other approaches that have been developed: domain decomposition  \cite{FLS94,NK07,CCFP12, F18}, max-plus algebras \cite{ME07,ME09, AGL08}, Model Predictive Control \cite{ GP11, G19}, Hopf-Lax representation formulas for Hamilton--Jacobi equations  \cite{DO16a, DO16b} and characteristics based methods \cite{YDG18}.

Recently deep learning techniques have been introduced for the approximation of the HJB equaitons. We mention that in \cite{DO16a, DO16b} has been proposed to apply a discrete version of Hopf-Lax representation formulas for Hamilton-Jacobi equations avoiding its global approximation on a grid. The advantage of this method is that it can
be applied at every point in the space and that it can be easily parallelized. However,
this method can not be used for general nonlinear control problems since the Hopf-Lax
representation formula is valid only for hamiltonians of the form $H(Du)$ whereas the
hamiltonians related to general optimal control problems are of the form $H(x; u;Du)$
(see e.g. \cite{BCD97}). Neural networks have also been used in e.g. \cite{HJW18} for general nonlinear PDEs with application to HJB equation too. Finally, we mention the recent papers \cite{HPBL21,BHLP22} where the authors present an hybrid approach based on Deep Neural Networks to solve stochastic control problems. They first approximate the optimal policy by means of neural networks and then the value function by Montecarlo regression.


 In the framework of optimal control problems an efficient acceleration technique based on the coupling between value and policy iterations has been proposed and studied in \cite{AFK15}. Although domain decomposition coupled with acceleration techniques can help to solve problems up to dimension $10$ we cannot solve problems beyond this limit with a direct approach (the recent application to a landing problem with state constraints in \cite{ABDZ18} shows the actual dimensional limitations of this approach). One way to attack high-dimensional problems is to apply  first a model order reduction technique (e.g. Proper Orthogonal Decomposition \cite{V13}) to have a low dimensional version of the dynamics by orthogonal projections. Thus, if the reduced system of coordinates for the dynamics has a reasonably low number of dimensions (e.g. $d\approx 5$) the problem can be solved via the DP approach. We refer to the pioneering  work on the coupling between model reduction and HJB approach \cite{KVX04} and to the recent work \cite{AFV17} that provides a-priori error estimates for the aforementioned coupling method. Another tentative has been made using a sparse grid approach in \cite{GK17}, there the authors apply HJB to the control of the wave equation and a spectral elements approximation in \cite{KK18} which allows to solve the HJB equation up to dimension $12$. More recently, in \cite{DKK19,DKS22} a tensor decomposition has been introduced to approximate the HJB equation.

In this paper we will analyze the method to mitigate the curse of dimensionality originally proposed in \cite{AFS18}, based on a tree structure algorithm which does not require a spatial discretization of the problem. In this way there is no need to store the nodes of the grid/triangulation of the computational domain.  The same time discretization has been proposed by Capuzzo Dolcetta in \cite{CD83} for the infinite horizon problem and it has been exploited in \cite{F84} for numerical purposes in combination with a grid projection via polynomial  interpolation. Here we are going to abandon the space grid interpolation to exploit  the tree structure based on the time discretization. Note that the tree strongly depends on the vector field, the number of steps used for the time discretization and the cardinality of the control set. Thus, the time discretization can result in a huge number of branches making the numerical approximation unfeasible. The crucial observation is that  not all the tree branches must be considered to get an accurate approximation of the value, with the help of a pruning  rule we can drastically reduce the complexity of the algorithm and finally solve the discrete time problem without projecting on a grid. The method applies to general nonlinear finite horizon optimal  control problems without additional assumptions on the structure of the dynamical system and it has been naturally extended  \cite{AFS18b}  to get high-order accurate approximations of the value function. Moreover, the TSA has been coupled in \cite{AS19} with a model order reduction technique based on Proper Orthogonal Decomposition (POD) to deal with optimal control problems driven by two dimensional nonlinear PDEs whose discretization produces a dynamical system of very high dimension (order of thousands). Finally, in \cite{AFS20} the authors present an extension of the TSA to problems with state constraints, together with a convergence result for the value function with convex constraints.

Our main contribution here is the precise error analysis of the TSA  developed in Section 3 and of the pruning technique. Note that the pruning technique is crucial to reduce the complexity of the algorithm and to attack problems in very-high dimension (a couple of examples are given here and others have been presented in \cite{AFS18}). We improve the order of convergence provided in \cite{FG99} for the finite horizon optimal control problem and we extend the error analysis to the pruned TSA keeping the same order of convergence under the semiconcavity assumption on the vector field and the costs.  Similar estimates for discrete time approximations of the infinite horizon optimal control problems can be found in \cite{CI84} where a first order approximation is proved under  semiconcavity assumptions, whereas high-order time approximations for the infinite horizon problem are presented in \cite{FF94}.

The paper is organized as follows: in Section~\ref{sec:tsa} we recall some basic facts about the discrete time approximation of the finite horizon problem via the DP approach and we present the construction of the tree-structure related to the controlled dynamics. Section~\ref{sec:eee} contains the main results,  in particular  a-priori error estimates for the first order approximation. A subsection is devoted to the analysis of the error for the pruning technique used to cut off the branches of the tree in order to reduce the global complexity of the algorithm (an extension to high-order time approximations is presented in \cite{AFS18b}).  Some numerical tests are presented and analyzed in Section~\ref{sec:test}. We give our conclusions and perspectives in Section \ref{sec:con}.

\section{Dynamic Programming on a Tree Structure}
\label{sec:tsa}
This section is devoted to the essential features of the dynamic programming approach and its numerical approximation. The interested reader will find in \cite{AFS18} more details on the tree structure algorithm.\\
Let us consider the classical {\it finite horizon problem}. Let the system be driven by
\begin{equation}\label{eq}
\left\{ \begin{array}{l}
\dot{y}(s)=f(y(s),u(s),s), \;\; s\in(t,T],\\
y(t)=x\in\R^d.
\end{array} \right.
\end{equation}
We will denote by $y:[t,T]\rightarrow\R^d$ the trajectory, by $f:\R^d\times\R^m\times[t,T]\rightarrow\R^d$ the vector field, by $u:[t,T]\rightarrow\R^m$ the control and by
\[\mathcal{U}=\{u:[t,T]\rightarrow U, \mbox{measurable} \}
\]
the set of admissible controls where $U\subset \R^m$ is a compact set. 
We assume that for each $u\in\mathcal{U}$ there exists a unique solution for \eqref{eq}. 

The cost functional for the finite horizon optimal control problem is given by 
\begin{equation}\label{cost}
 J_{x,t}(u):=\int_t^T L(y(s,u),u(s),s)e^{-\lambda (s-t)}\, ds+g(y(T))e^{-\lambda (T-t)},
\end{equation}
where $L:\R^d\times\R^m\times [t,T]\rightarrow\R$ is the running cost and $\lambda\geq0$ is the discount factor. In the present work we will assume that the functions $f,L$ and $g$ are continuous in all the variables and bounded:
 \begin{align}
 \begin{aligned}\label{Mf}
|f(x,u,s)|& \le M_f,\quad |L(x,u,s)| \le M_L,\quad |g(x)| \le M_g, \cr
&\forall\, x \in \mathbb{R}^d, u \in U \subset \mathbb{R}^m, s \in [t,T], 
\end{aligned}
\end{align}
the functions $f$ and $L$ are Lipschitz-continuous with respect to the first variable
\begin{align}
\begin{aligned}\label{Lf}
&|f(x,u,s)-f(y,u,s)| \le L_f |x-y|, \quad |L(x,u,s)-L(y,u,s)| \le L_L |x-y|,\cr
&\qquad\qquad\qquad\qquad\forall \, x,y \in \mathbb{R}^d, u \in U \subset \mathbb{R}^m, s \in [t,T], 
\end{aligned}
\end{align}
%
and the cost $g$ is also Lipschitz-continuous:
\begin{equation}
|g(x)-g(y)| \le L_g |x-y|, \quad \forall x,y \in \mathbb{R}^d.
\label{Lg}
\end{equation}
In the sequel, we will also need to assume the semiconcavity of the functions $L$ and $g$:
\begin{equation}\label{L_con}
L(x+h, t+ \tau, u) -2L(x,t,u) + L(x-h,t-\tau,u)\leq C_L(|h|^2+\tau^2), \quad \forall x \in \mathbb{R}^d,\forall h, \tau >0,
\end{equation}
\begin{equation}\label{g_con}
g(x+h) -2g(x) + g(x-h)\leq C_g|h|^2,\quad \forall x \in \mathbb{R}^d,\forall h>0 \,,
\end{equation}
and a stronger assumption for the function $f$:
\begin{equation}\label{Cf}
|f(x+z,u,t+\tau)-2f(x,u,t)+f(x-z,u,t-\tau) | \le C_f (|z|^2+ \tau^2),\quad
 \forall u \in U, \;\forall x,z \in \mathbb{R}^d,\, \forall t,\tau >0\,.
\end{equation}
Conditions \eqref{L_con}-\eqref{g_con} provide an upper bound for a discrete version of the second order derivative. Furthermore, it can be proved that \eqref{L_con} and \eqref{g_con} are equivalent to the boundedness of the second order derivative in the sense of the distribution (we refer to \cite{CS04} for further details). Condition \eqref{Cf} is a stronger hypothesis since it adds a lower bound to the previous ones.

To derive optimality conditions we use the well-known Dynamic Programming Principle (DPP) due to Bellman. We first define the value function for an initial condition $(x,t)\in\R^d\times [t,T]$:
\begin{equation}
v(x,t):=\inf\limits_{u\in\mathcal{U}} J_{x,t}(u)
\label{value_fun}
\end{equation}
which satisfies the DPP, i.e. for every $\tau\in [t,T]$:
\begin{equation}\label{dpp}
v(x,t)=\inf_{u\in\mathcal{U}}\left\{\int_t^\tau L(y(s),u(s),s) e^{-\lambda (s-t)}ds+ v(y(\tau),\tau) e^{-\lambda (\tau-t)}\right\}.
\end{equation}
Due to \eqref{dpp} we can derive the HJB equation for every $x\in\R^d$, $s\in [t,T)$: 
\begin{equation}\label{HJB}
\left\{
\begin{array}{ll} 
&-\dfrac{\partial v}{\partial s}(x,s) +\lambda v(x,s)+ \max\limits_{u\in U }\left\{-L(x, u,s)- \nabla v(x,s) \cdot f(x,u,s)\right\} = 0, \\
&v(x,T) = g(x).
\end{array}
\right.
\end{equation}
Once the value function is known, by e.g. \eqref{HJB}, then it is possible to compute the optimal feedback control as:
\begin{equation}\label{feedback}
u^*(t):=  \argmax_{u\in U }\left\{-L(x,u,t)- \nabla v(x,t) \cdot f(x,u,t)\right\},
\end{equation}
where \eqref{feedback} has to be understood in an a.e. sense because viscosity solutions are Lipschitz continuous (see e.g. \cite{BCD97}).
\subsection{Numerical approximation for HJB equation on a tree structure}\label{sec2.1}
Equation \eqref{HJB} is a nonlinear PDE of the first order which is hard to solve analitically. However, several numerical methods, such as e.g. finite difference or semi-Lagrangian schemes, are available to approximate the solution. 
In the present work we recall the semi-Lagrangian method on a tree structure based on the recent work \cite{AFS18}. Fixed $\overline{N}$ as the number of temporal time steps, we introduce the semi-discrete problem with a time step $\Delta t: = (T-t)/\overline N$:
\begin{equation}
\left\{\begin{array}{ll}\label{SL}
V^{n}(x)=\min\limits_{u\in U}\left\{\Delta t\,L(x, u, t_n)+e^{-\lambda \Delta t}V^{n+1}(x+\Delta t f(x, u, t_n))\right\},\qquad n= \overline{N}-1,\dots, 0,\\
V^{\overline{N}}(x)=g(x),\qquad\qquad\quad\qquad\qquad\qquad\qquad\qquad\qquad\qquad\qquad\qquad x \in \R^d,
\end{array}\right.
\end{equation}
where $t_n=t+n \Delta t,\, t_{\overline N} = T$ and $V^n(x):=V(x, t_n).$ For the sake of completeness we would like to mention that a fully discrete approach is typically based on a time discretization which is projected on a fixed state-space grid of the numerical domain, see e.g. \cite{FG99}. The current work aims to provide error estimates for the algorithm proposed in \cite{AFS18}. 

For the reader's convenience we recall the tree structure algorithm. Let us consider a finite number of admissible controls $\{u_1,...,u_M \}$, obtained discretizing the control domain $U\subset\mathbb{R}^m$ with step-size $\Delta u$. A typical example is when $U$ is an hypercube, which is discretized in all the directions with constant step-size $\Delta u$, obtaining the finite set $U^{\Delta u}=\{u_1,...,u_M \}$. To simplify the notations in the sequel we continue to denote by $U$ the discrete set of controls. We will denote the tree by $\mathcal{T}:=\cup_{j=0}^{\overline{N}} \mathcal{T}^j,$ where each $\mathcal{T}^j$ contains the nodes of the tree correspondent to time $t_j$. The first level $\mathcal{T}^0 = \{x\}$ is clearly given by the initial condition $x$. Starting from the initial condition $x$, we compute all the nodes given by the dynamics \eqref{eq} discretized by e.g. an explicit Euler scheme with different discrete controls $u_j \in U $
$$\zeta_j^1 = x+ \Delta t \, f(x,u_j,t_0),\qquad j=1,\ldots,M.$$ Therefore, we have $\mathcal{T}^1 =\{\zeta_1^1,\ldots, \zeta^1_M\}$. We observe that all the nodes can be characterized by their $n-$th {\em time level}, as follows
$$\mathcal{T}^n = \{ \zeta^{n-1}_i + \Delta t f(\zeta^{n-1}_i, u_j,t_{n-1}),\, j=1, \ldots, M,\,i = 1,\ldots, M^{n-1}\}.$$

We are considering an Euler approximation of the dynamical system to simplify the presentation, but the algorithm can be extended to high-order approximations, as illustrated in \cite{AFS18b}.
All the nodes of the tree can be briefly denoted as
$$\mathcal{T}:= \{ \zeta_j^n,\,  j=1, \ldots, M^n,\, n=0,\ldots, \overline{N}\},$$ 
where the nodes $\zeta^n_i$ are the results of the discrete dynamics at time $t_n$ with the controls $\{u_{j_k}\}_{k=0}^{n-1}$:
$$\zeta_{i_n}^n = \zeta_{i_{n-1}}^{n-1} + \Delta t f(\zeta_{i_{n-1}}^{n-1}, u_{j_{n-1}},t_{n-1})= x+ \Delta t \sum_{k=0}^{n-1} f(\zeta^k_{i_k}, u_{j_k},t_k), $$
with $\zeta^0 = x$, $i_k = \floor*{\dfrac{i_{k+1}}{M}}$ and $j_k\equiv i_{k+1} \mbox{mod } M$ and $\zeta_i^k \in \R^d, i=1,\ldots, M^k$. The notation $\floor*{\cdot}$ represents the floor function.

Although theoretically the tree structure allows to solve high dimensional problems, its construction might be expensive due to the huge amount of memory allocations, since $ \mathcal{T} =O( M^{\overline{N}}),$
where $M$ is the number of controls and $\overline{N}$ the number of time steps.
For this reason we are going to introduce the following pruning criterion:
two given nodes $\zeta^n_i$ and $\zeta^n_j$ will be merged if 
\begin{equation}\label{tol_cri}
\Vert \zeta^n_i-\zeta^n_j \Vert \le \ep, \quad \mbox{ with }i\ne j \mbox{ and } n = 0,\ldots, \overline{N}, 
\end{equation}
for a given threshold $\ep>0$. Criterion \eqref{tol_cri} will be useful in order to save a huge amount of memory. The selection is made on the fly and in case of two nodes within a distance $\ep$, we keep the first node computed, e.g. if $\zeta^n_i$ and $\zeta^n_j$ satisfy \eqref{tol_cri} with $i<j$, we will neglect the node $\zeta^n_j$. Later, we will show a result on the threshold $\ep>0$ in order to guarantee first order convergence.

Once the tree $\mathcal{T}$ has been built, the numerical value function $V(x,t)$ will be computed on the tree nodes and we will denote by $V^n(x)$ the value function computed in $x$ at time $t_n=t+ n \Delta t$. 
The computation of the value function is now straightforward. The TSA defines a time dependent grid $\mathcal{T}^n=\{\zeta^n_j\}_{j=1}^{M^n}$ for $n=0,\ldots, \overline{N}$ and
we can approximate \eqref{dpp} as follows: 
\begin{equation}
\begin{cases}
V^{n}(\zeta^n_i)= \min\limits_{u\in U} \{e^{-\lambda \Delta t} V^{n+1}(\zeta^n_i+\Delta t f(\zeta^n_i,u,t_n)) +\Delta t \, L(\zeta^n_i,u,t_n) \}, \qquad \zeta^n_i \in \mathcal{T}^n\,, n = \overline{N}-1,\ldots, 0, \\
V^{\overline{N}}(\zeta^{\overline{N}}_i)= g(\zeta_i^{\overline{N}}), \qquad\qquad\quad\qquad\qquad\qquad\qquad\qquad \qquad\qquad \qquad\qquad   \zeta_i^{\overline{N}} \in \mathcal{T}^{\overline{N}}.
\end{cases}
\label{HJBt2}
\end{equation}

The minimization in \eqref{HJBt2} is computed by comparison on the set of discrete controls $U$. 

Once the value function is computed on the nodes of the tree, it is possible to consider a post-processing procedure to achieve a feedback control. We consider the formula for the synthesis of the feedback control
$$
u_*^n(x)=  \argmin_{u\in U }\{e^{-\lambda \Delta t} I_{\mathcal{T}^{n+1}}[V^{n+1}](x+\Delta t f(x,u,t_n)) +\Delta t \, L(x,u,t_n) \}
$$
where $I_{\mathcal{T}^{n+1}}[V^{n+1}]$ is an interpolation operator based on the scattered data $\left(\mathcal{T}^{n+1},V^{n+1}(\mathcal{T}^{n+1})\right)$. A detailed analysis of the computational methods to approximate the feedback control goes beyond the scopes of this work. The interested reader will find further information about this procedure in Chapter 3.1 in \cite{S20} or \cite{AS22}. We note that in general it is possible to obtain a control in feedback form on the nodes of tree without the application of interpolation operators.
Finally, we would like to mention that a detailed description and comparison about the classical method and tree structure algorithm can be found in \cite{AFS18}.

\begin{rmk}[Efficient Pruning]
\label{efficient}
The pruning criterion \eqref{tol_cri} could result in a very expensive algorithm, especially when we deal with high dimensional dynamics. This is the case of a semidiscretization of an evolutive PDE where the dimension easily reaches thousands of unknown variables. To speed up the pruning in this case, we will consider an orthogonal projection $\mathcal{P}$ of the data onto a lower dimensional space which can capture the main features of the dynamics. This can be obtained, for instance, by a Singular Value Decomposition of a matrix containing some snapshots taken from a coarse approximation of the tree. This turns out to accelerate the algorithm, as shown in Section \ref{sec:test}. We refer to \cite{AFS18,AS19} for more details on this technique in different contexts. Since the operator $\mathcal{P}$ is a non-expansive operator, $i.e.$
\begin{equation}
\Vert \mathcal{P} x -\mathcal{P} y \Vert \le \Vert x - y \Vert \quad \forall x,y \in \mathbb{R}^d.
\label{boundP}
\end{equation}
Hence, if two nodes satisfy the pruning criterion \eqref{tol_cri}, their projections satisfy the same condition due to the inequality \eqref{boundP} . Thus, it will be sufficient to collect all the nodes fulfilling the pruning criterion in the projected space and to check the same criterion in the original dimension for the selected nodes.
\end{rmk}

\section{Error estimate for TSA with Euler discretization}
\label{sec:eee}
In this section we will provide an error analysis for the TSA. We denote $y(s)$ as the exact continuous solution for \eqref{eq} and whenever we want to stress the dependence on the control $u$, the initial condition $x$ and initial time $t$ we write $y(s; u, x, t)$. We further define $y^n(u)$ as its numerical approximation by an explicit Euler scheme at time $t_n$. We will consider the piecewise constant extension $\tilde{y}(s;u)$ of the approximation such that

\begin{equation}\label{apy}
\tilde{y}(s,u):={y}^{[s/ \Delta t]}(u),\qquad s \in [t,T], u\in \mathcal{U}^\Delta,
\end{equation}
where $[\cdot]$ stands for the integer part and
$$
\mathcal{U}^\Delta =\{ u:[t,T) \rightarrow U, \mbox{ such that } u(s)= \sum_{k=n}^{\overline{N}-1} u^k \chi_{[t_k,t_{k+1})}(s) \}.
$$
 Let us now consider the discretized version of the cost functional \eqref{cost}:
\begin{align*}
\begin{aligned}
&J^{\Delta t}_{x,s}(u)=(t_{n+1}-s) L(x, u, s)+ \Delta t \sum_{k=n+1}^{\overline{N}-1} L(y^k, u^k, t_k) e^{-\lambda (t_k-s)} + g(y^{\overline{N}})e^{-\lambda (t_N-s)}\\
&= \int_{s}^T L(\tilde{y}(\sigma,u;x,s),u(\sigma),{\floor*{\frac{\sigma}{\Delta t}}}\Delta t) e^{-\lambda \left({\floor*{\frac{\sigma}{\Delta t}}}\Delta t-s\right)} d\sigma + g\left(\tilde{y}(T,u;x,s)\right)e^{-\lambda (T-s)}
\end{aligned}
\end{align*}
for $s\in [t_n,t_{n+1})$. We define the discrete value function as
$$
V(x,t):=\inf_{u \in \mathcal{U}^\Delta} J_{x,t}^{\Delta t}(u)
$$
which can be computed by the backward problem 
\begin{align}\label{HJBt3}
\begin{aligned}
&V(x,s) = \min\limits_{u\in U} \{e^{-\lambda (t_{n+1}-s)} V(x+(t_{n+1}-s) f(x,u,s), t_{n+1}) + (t_{n+1}-s) \, L(x,u,s) \}, \\
&V(x,T) = g(x), \hspace{6cm} x \in \mathbb{R}^d, s \in [t_n,t_{n+1}).
\end{aligned}
\end{align}

The aim of this section is to find a priori error estimates for the tree algorithm and show the rate of convergence of the approximation $V$. We will show that if the dynamics is discretized by forward Euler method the error is $O(\Delta t)$:
\begin{equation}\label{eq:res}
\sup_{(x,t)\in\R^d \times [0,T]}| v(x,t) - V(x,t) | \le \widehat C (T) \Delta t
\end{equation}
where $\Delta t$ is the time discretization of \eqref{eq} and $v$ is the exact solution \eqref{value_fun}. We remark that the estimate guarantees the same order of convergence of the discretization scheme for the dynamical system \eqref{eq}. To simplify the proof of the main result \eqref{eq:res}  we have splitted the proof into two parts (see Theorem \ref{unverso} and Theorem \ref{secondoverso}). We note that this result improves the estimate in \cite{FG99} under the semiconcavity assumption and it is in line with a similar result for the infinite horizon problem in \cite{CI84}.
To begin with, we show some estimates for the Euler scheme which will be useful to prove the error estimates for TSA. The proposition below follows directly from Gr\"onwall's lemma and its discrete version.
\begin{proposition}\label{prop1}
Let us consider the exact solution trajectory $y(s; u, x, t)$ and its approximation $\tilde{y}(s; u, x, t)$ of \eqref{eq} for a given admissible control $u \in  \mathcal{U}^\Delta$. Furthermore, let us assume that assumptions \eqref{Mf} and \eqref{Lf} hold true. We then obtain the following estimates applying the Euler scheme to \eqref{eq}:
\begin{align}
|y(s; u, x, t)-\tilde{y}(s; u, x, t)| \le M_f \Delta t e^{L_f(s-t)},\label{euler1}\\
\begin{aligned}\label{euler2}
|\tilde{y}(s; u, x+z, t+\tau)-\tilde{y}(s; u, x, t)| \le (|z|+M_f \tau) (1+L_f \Delta t)^{n-k} \cr
 s \in [t_n, t_{n+1}) \mbox{ and }t+\tau\in [t_k,t_{k+1}) \mbox{ with }\tau \ge 0, \quad s \ge t + \tau.
 \end{aligned}
\end{align}
\end{proposition}
Using Proposition \ref{prop1}, we are able to prove one side of \eqref{eq:res} as shown in the following theorem.
\begin{theorem}\label{unverso}
Let us assume that conditions \eqref{Mf},\eqref{Lf} and \eqref{Lg} hold true. Then 
\begin{equation}\label{1est}
\sup_{(x,t)\in \mathbb{R}^d \times [0,T]} (v(t,x)-V(t,x)) \le C(T) \Delta t, \quad \forall t \in [0,T],
\end{equation}
where $C(T)$ is a constant which does not depend on the time step $\Delta t$.
\end{theorem}
{\em Proof.}
First, since $\mathcal{U}^\Delta \subset \mathcal{U}$, we have
$$
v(t,x)-V(t,x) \le  \inf_{u \in \mathcal{U}^\Delta} J_{x,t}(u) - \inf_{u \in \mathcal{U}^\Delta} J_{x,t}^{\Delta t}(u) \le \sup_{u \in \mathcal{U}^\Delta} |J_{x,t}(u)-J_{x,t}^{\Delta t}(u)|.
$$
For a given control $u  \in \mathcal{U}^\Delta$, we use the assumptions in Proposition \ref{prop1} to obtain the following
\begin{align*}
 \left|J_{x,t}(u)-J_{x,t}^{\Delta t}(u)\right| &\le  M_f \Delta t \left(\frac{L_L }{L_f}e^{L_f T}+ L_g e^{L_fT} \right). 
\end{align*}
Then, we obtain the desired estimate \eqref{1est} with $C(T)= M_f e^{L_fT} \left(  \frac{L_L}{L_f}+L_g \right).\Box$

To prove the remaining side of \eqref{eq:res} we need to assume the semiconcavity of the functions $g,L$ and a stronger assumption on $f$. The proof of  Theorem \ref{secondoverso} is based on some technical lemmas that are presented below.
\begin{proposition}\label{prop2}
Let us consider the assumptions of Proposition \ref{prop1} and consider the function $f(x,u,t)$ as a Lipschitz-continuous function in time and space uniformly in $u$ satisfying \eqref{Cf},
%
then
\begin{align}
\begin{aligned}\label{euler3}
&|\tilde{y}(s; u, x+z,t+\tau)-2\tilde{y}(s; u, x,t)+\tilde{y}(s; u, x-z,t-\tau)| \le \widetilde{C}(T)(|z|^2+\tau^2),\cr
&\qquad\qquad\qquad\qquad\qquad\qquad \forall s \ge t+ \tau, \forall u \in U, \;\forall x,z \in \mathbb{R}^d,\, \forall t,\tau >0, 
\end{aligned}
\end{align}
where $\widetilde{C}(T)$ is a constant that depends on $T$ but does not depend on the time step $\Delta t$.
\end{proposition}
{\em Proof.}
Let us suppose that $t+ \tau \in [t_k,t_{k+1})$ for some $k>0$, $t \in [t_0,t_1)$ and $t-\tau \in [t_{-k-1},t_{-k})$. Let us consider $s\in[t_{n+1},t_{n+2})$, to ease the notation we will denote 
\begin{align*}
\tilde{y}(s,u, x+z, t+\tau)&:=y^{n+1}_+, &\tilde{y}(s,u,x,t)&:=y^{n+1},\\
\tilde{y}(s,u, x-z, t-\tau)&:=y^{n+1}_{-}, &\qquad f(y,u,t_n)&:=f^n(y),
\end{align*}
and we will drop the dependence on the control $u$ since it is fixed for all the terms considered above.
Applying only one step of the forward Euler scheme with $n \ge k$ we get
\begin{align*}
y^{n+1}_+ - 2y^{n+1} +y^{n+1}_{-}=y^{n}_+ -2y^{n}+y^{n}_{-}+ \Delta t \left(f^n(y^{n}_+) -2f^n(y^{n}) +f^n(y^{n}_{-})\right).
\end{align*}
Thus, from assumption \eqref{Cf} we obtain the following 
\begin{align*}
&|f^n(y^{n}_+)- 2f^n(y^{n})+f^n(y^{n}_{-})| =\\
& |f^n(y^{n}_+)-2f^n(y^{n})+(f^n(y^n- (y^n_+ -y^n))-f^n(y^n- (y^n_+ -y^n))) +f^n(y^{n}_{-})| \le\\
&| f^n(y^n+ (y^n_+-y^n))-2 f^n(y^n)+ f^n(y^n- (y^n_+ -y^n))|+\\
& |f^n(y^n_{-})-f^n(y^n- (y^n_+ -y^n))|\le C_f|y^n_+ -y^n|^2+L_f \left|y^n_{+} -2y^n +y^n_{-}\right|.
\end{align*}
Then, applying \eqref{euler2} we obtain 
\begin{equation}\label{est1}
|y^{n+1}_+ -2y^{n+1} +y^{n+1}_{-}| \le \Delta t C_1 C_2^{2(n-k)}+ C_2|y^n_+ -2y^n +y^n_{-}|, 
\end{equation}
with $C_1=C_f(|z|+M_f \tau)^2$ and $C_2=1+L_f \Delta t$. Then, iterating \eqref{est1} we obtain
\begin{equation}
|y^{n+1}_+ -2y^{n+1} +y^{n+1}_{-}| \le \Delta t C_1C_2^{2(n-k)}\sum_{j=0}^{n-k} C_2^{-j}+C_2^{n-k+1}| x+z -2y^k+y^k_{-}|.
\label{induction}
\end{equation}
Writing the full discrete dynamics for $y^k$ and $y^k_{-}$, the right hand side in \eqref{induction} becomes
\begin{align}
\begin{aligned}\label{estimate2}
&\Delta t C_1C_2^{2(n-k)} \frac{1-C_2^{-(n-k+1)}}{1-C_2^{-1}}+ C_2^{n-k+1}\Delta t\left|-2 \sum_{j=0}^{k-1} f^j(y^j)+\sum_{j=-k}^{k-1} f^j(y^j_{-})\right| \le\\
&\frac{C_1C_2^{2(n-k)+1}}{L_f} +C_2^{n-k+1} \Delta t \left| \sum_{j=0}^{k-1} \left( f^j(y^j_{-})-f^j(y^j)+ f^{j-k}(y^{j-k}_{-})-f^{j}(y^j) \right) \right|.
\end{aligned}
\end{align}
Now we want to estimate last term in \eqref{estimate2}. Since the first term $f^j(y^j_{-})-f^j(y^j)$ of the sum can be obtained as a particular case of the second one, with $k=0$, let us now focus on the last term
\begin{equation}
\left| \sum_{j=0}^{k-1} \left(f^{j-k}(y^{j-k}_{-})-f^{j}(y^j)\right)\right|\le   L_f  \sum_{j=0}^{k-1}\left( \left|y^{j-k}_{-}-y^j\right|+ \tau\right).
\end{equation}
Using \eqref{euler2}, we can write
$$
\left|y^{j-k}_{-}-y^j\right| \le \left|y^{j-k}_{-}-y^{j}_{-}\right|+\left|y^{j}_{-}-y^j\right|\le \tau M_f + (|z|+M_f\tau)C_2^{j}.
$$
Finally we get 
$$
|y^{n+1}_+ - 2y^{n+1} +y^{n+1}_{-}|\le
 \frac{C_1C_2^{2(n-k)+1}}{L_f}+2C_2^{n-k+1}L_f\left( \tau^2 M_f +C_2^k\tau(|z|+M_f \tau) \right).
$$

Noting that $C_2^n=(1+L_f\Delta t)^n \le e^{t_nL_f}$, we obtain the desired result with the constant $\widetilde{C}(T)$ equal to
\begin{equation}
\widetilde{C}(T)=2 e^{2T} \left(\frac{C_f (\max\{1,M_f\})^2}{L_f} +L_f(2M_f+1) \right).\Box
\end{equation}

Let us recall some properties for the scheme \eqref{HJBt3} which will be useful later since the reverse inequality in \eqref{1est} needs the assumption of semiconcavity for the numerical approximation $V$. We refer to \cite{CS04} for a detailed discussion on the importance of the semiconcavity in control problems. 
\begin{proposition}\label{semicon} 
Let us suppose that the functions $L$ and $g$ are both Lipschitz-continuous and satisfy the semiconcavity assumptions \eqref{L_con} and \eqref{g_con}. Furthermore, let us consider the function $f(x,u,t)$ as a Lipschitz-continuous function in time and space uniformly in $u$ such that it satisfies \eqref{Cf}.
Then the numerical solution $V$ is semiconcave:
\begin{equation}\label{FV:con}
V(x+z,t+\tau)-2V(x,t)+V(x-z,t-\tau) \le C_V (|z|^2+ \tau^2) \quad \forall x,z \in \mathbb{R}^n, t,\tau \ge 0.
\end{equation}
\end{proposition}
{\em Proof.}
Given $x,z \in \mathbb{R}^n$ and $t,\tau \in [0,T]$ such that $t+ \tau \in [t_k,t_{k+1})$, $t \in [t_0,t_1)$ and $t-\tau \in [t_{-k-1},t_{-k})$, we need to prove \eqref{FV:con}.
By the definition of value function, we can write
\begin{align}
\begin{aligned}
&V(x+z,t+\tau)+V(x-z,t-\tau)-2V(x,t) \le \sup_{u \in U } \left\{ (t_{k+1}-t-\tau) L(x+z,t+\tau,u) +\right.\\
&\left. (t_{-k}-t+\tau) L(x-z,t-\tau,u)- 2(t_1-t) L(x,t,u) \right\} \\
 &+\sup_{u \in \mathcal{U}^{\Delta} }\left(J^{\Delta t}_T(x+z,t_{k+1},u)+J^{\Delta t}_T(x-z,t_{-k},u)-2J^{\Delta t}_T(x,t_1,u)\right).
\end{aligned}
\end{align}
We can estimate the first term on the right hand side as follows
\begin{align}
\begin{aligned}
 &(t_{k+1}-t-\tau) L(x+z,t+\tau,u) + (t_{-k}-t+\tau) L(x-z,t-\tau,u)- 2(t_1-t) L(x,t,u) )  \\
 &\le \Delta t \max\{L(x+z,t+\tau,u) + L(x-z,t-\tau,u)-  2L(x,t,u),0\} .
\end{aligned}
\end{align}
Without loss of generality, we will consider $\lambda=0$. Given $u  \in \mathcal{U}^{\Delta}$ and denoted by $L(y,u,t_n) = L^n(y)$, we have that the remaining right hand side is equal to
\begin{align}
\begin{aligned}\label{sums}
&\Delta t\left( \sum_{n=k+1}^{N-1}\left( L^n(y^n_+)+ L^n(y^n_{-}) -2 L^n(y^n)\right) + \sum_{n=1}^{k}\left( L^n(y^n_{-})-2 L^n(y^n)\right)\right)+\\
&\Delta t\left( \sum_{n=-k}^{0} L^n(y^n_{-}) \right)+ g(y^N_+)+g(y^N_{-})-2g(y^N).
\end{aligned}
\end{align}
As already done in the proof of Proposition \ref{prop2}, exploiting the properties of $L$, i.e. Lipschitz-continuity and semiconcavity with constant $C_L>0$, for the first summation in \eqref{sums} we have:
\begin{equation} \label{dis}
 L^n(y^n_+)+ L^n(y^n_{-})-2 L^n(y^n) \le C_L|y^n_+-y^n|^2+L_L|y^n_+ -2y^n+y^n_{-}|.
\end{equation}
Using \eqref{euler2}, we obtain the following bound for the first term
\begin{align}
\begin{aligned}
&\Delta t C_L \sum_{n=k+1}^{N-1} |y^n_+ - y^n|^2 \le \Delta t C_L (|z|+M_f \tau)^2 \sum_{n=k+1}^{N-1} (1+L_f\Delta t)^{2(N-k)}\le \\
&\le \frac{C_L}{L_f}(1+L_f \Delta t)^{2N} (|z|+M_f \tau)^2 \le 2\max\{M_f^2,1\}\frac{ C_L}{L_f} e^{2L_fT} (|z|^2+\tau^2).
\end{aligned}
\end{align}
Using \eqref{euler3}, we obtain directly
$$
\Delta t  L_L \sum_{n=k+1}^{N-1}|y^n(x+z,t+\tau)-2y^n(x,t)+y^n(x-z,t-\tau)| \le T L_L \widetilde{C}(T) (|z|^2+\tau^2).
$$ Finally we rewrite the second and third summation in \eqref{sums} in the following way
$$
\Delta t \sum_{n=1}^{k} [(L^n(y^n_{-})- L^n(y^n))+ (L^{n-k-1}(y^n_{-})- L^n(y^n))]
$$
and with the same procedure used in the proof of Proposition \ref{prop2} and applying \eqref{dis} with $g$, we obtain the desired estimate.$\Box$

Next, we introduce a further characterization of $V$ which will turn out to be useful to prove Theorem \ref{secondoverso}. The proof of this proposition can be found in \cite{FG99}.
\begin{proposition}
Assume that assumptions \eqref{Mf}, \eqref{Lf}, \eqref{Lg} hold true. Then the solution $V$ of \eqref{HJBt2} is bounded (and uniformly continuous). Furthermore, the following estimate holds
\begin{equation}
|V(y_0,s)-V(x_0,T)| \le C\left( |y_0-x_0| + (T-t_n)+ \Delta t \right), \quad s \in [t_n,t_{n+1}), \forall x_0,y_0 \in \mathbb{R}^n.
\label{estimatepsi}
\end{equation}
\end{proposition}

Finally, before proving Theorem \ref{secondoverso}, we introduce the following lemma (proved in \cite[Lemma 4.2, p. 170 ]{CI84}).
\begin{lemma}\label{lemma1}
Let $\xi: \mathbb{R}^n \times [0,T] \rightarrow \mathbb{R}$ satisfy
\begin{equation*}
\xi(y+z,t+\tau)-2\xi(y,t)+\xi(y-z,t-\tau) \le C_{\xi} \left(|z|^{2} + |\tau|^{2}\right),
\end{equation*}

$\forall y,z \in \mathbb{R}^n$, $\forall t,\tau\in [0,T]$ such that $t+\tau,t,t-\tau \in [0,T]$ and
\begin{equation*}
\xi(0,0)=0\;, \quad \limsup_{(y,t) \rightarrow (0,0)} \frac{\xi(y,t)}{|y|+|t|} \le 0.
\end{equation*}
Then
$$
\xi(y,t) \le \frac{C_{\xi}}{6} (|y|^{2}+|t|^{2}) \quad \forall y \in \mathbb{R}^n, t \in [0,T].
$$
\end{lemma}
%
%
%
%
%

We are now able to prove our main result.

\begin{theorem}\label{secondoverso}
Let the assumptions \eqref{Mf}-\eqref{Lf}-\eqref{Lg} hold true. Moreover, let us assume that the functions $L$ and $g$ are semiconcave and that the function $f(x,u,t)$ is Lipschitz continuous in space and time uniformly in $u$ and it satisfies \eqref{Cf}.
Then
\begin{equation}
\sup_{(x,t)\in \mathbb{R}^d \times [0,T]} \left(V(t,x)-v(t,x)\right) \le \overline C(T) \Delta t\;, \quad \forall t \in [0,T].
\label{2est}
\end{equation}
\end{theorem}
{\em Proof.}
The first part of the proof follows closely from \cite{FG99}. We introduce the auxiliary function
$$
\phi(y,t,x,s)= V(y,t)-v(x,s)+\beta_{\epsilon}(x-y) + \eta_{\alpha}(t-s),
$$
where $\beta_\epsilon(x)=-\frac{|x|^2}{\epsilon^2}$ and $\eta_\alpha(s)=-\frac{s^2}{\alpha^2}$. 

Since $v$ and $V$ are bounded, then for any $\delta>0$, there exist $(y_1,\tau_1),(x_1, s_1)$ such that
$$
\phi(y_1,\tau_1,x_1, s_1)> \sup \phi-\delta.
$$

Choosing $\theta(y,x) \in C^\infty_0(\mathbb{R}^d \times \mathbb{R}^d)$, with $\theta(y_1,x_1)=1$ and $0 \le \theta \le 1$, $|D\theta| \le 1$, such that for any $\delta \in (0,1)$, 
$$
\zeta(y,t,x,s)=\phi(y,t,x,s)+ \delta \theta(y,x)
$$
 has a maximum point $(y_0,\tau_0,x_0,s_0)$, with $y_0,x_0\in \supp  \theta$  and $\tau_0, s_0 \in [0,T].$ Therefore, if we set
$$
\Phi(x,s)=V(y_0,\tau_0)+\beta_{\epsilon}(y_0-x) + \eta_{\alpha}(\tau_0-s) + \delta \theta(y_0,x),
$$
we can observe that $(x_0,s_0)$ is a local min for $v(x,s)-\Phi(x,s)$. By definition of $\zeta$, we have that
\begin{align}
\begin{aligned}\label{dis2}
& V(y_0,\tau_0)-v(x_0,s_0)+\beta_{\epsilon}(y_0-x_0) + \eta_{\alpha}(\tau_0-s_0) +\delta \theta(y_0,x_0) \ge\\
&\ge   V(y,t)-v(x,s)+\beta_{\epsilon}(y-x) + \eta_{\alpha}(t-s) +\delta \theta(y,x).
\end{aligned}
\end{align}
From \eqref{dis2} with $x=y=y_0$, $s=s_0$ and $t=\tau_0$,  we get
\begin{equation}\label{y0x0}
|y_0-x_0| \le \epsilon^2(L_v+ \delta),
\end{equation}
and similarly, with $x=x_0,$ $y=y_0$ and $s=t=\tau_0$:
\begin{equation}\label{s0t0}
 |s_0-\tau_0| \le \alpha^2 L_v,
\end{equation}

where $L_v$ is the Lipschitz constant of $v$ with respect to time and space.
Using \eqref{dis2}, \eqref{y0x0} and \eqref{s0t0}, we obtain
\begin{equation}
V(x,s)-v(x,s) \le V(y_0,\tau_0)-v(x_0,s_0)+(L_v+\delta) \epsilon^2+ \alpha^2 L_v+ 2 \delta.
\label{V_v}
\end{equation}
Let us now consider three cases as suggested in \cite{FG99}. We recall that in this theorem we improve their approximation by means of the semiconcavity which turns out to be essential in the third case of the proof. However, in the first two cases we can directly obtain first order convergence. Without this property we can only prove an order of convergence of $\frac{1}{2}$.

\medskip
\paragraph{First case ($\tau_0=T$)}
In this case $V(y_0,T)=g(y_0)=v(y_0,T)$. Thus, using the Lipschitz-continuity of $g$ we obtain the desired result, setting $\alpha=\epsilon= \sqrt{\Delta t}$.

\medskip
\paragraph{Second case ($\tau_0 \neq T$, $s_0=T$)}
In this case $v(x_0,T)=g(x_0)=V(x_0,T)$. Supposing $\tau_0 \in [t_n,t_{n+1})$ and using the estimate \eqref{estimatepsi} in \eqref{V_v}, we obtain
$$
V(x,s)-v(x,s) \le C\left( |y_0-x_0| + (T-t_n)+ \Delta t \right) +(L_v+\delta) \epsilon^2+ \alpha^2 L_v+ 2 \delta.
$$
Since $\tau_0-t_n\le \Delta t$, using \eqref{s0t0} we can write that 
$$
T-t_n \le L_v \alpha^2+ \Delta t,
$$
and using \eqref{y0x0}, finally we get
$$
V(x,s)-v(x,s) \le C_3 \epsilon^2+ C_4 \alpha^2+ C_5 \Delta t+ 2 \delta.
$$
If we set $\alpha=\epsilon= \sqrt{\Delta t}$, we get the result, since $\delta$ is arbitrary.

\medskip
\paragraph{Third case ($\tau_0,s_0 \neq T$)}
We know that $v$ is a viscosity solution, this means that there exists a control $ u^* \in U$ such that
$$
-\partial_s \Phi(x_0,s_0)+ \lambda v(x_0,s_0) -f(x_0,s_0,u^*) \cdot \nabla_x \Phi(x_0,s_0) - L(x_0,s_0,u^*) \ge 0.
$$ 
Thus, we obtain

\begin{equation}\label{2}
\nabla \eta_\alpha(\tau_0-s_0)+ \lambda v(x_0,s_0)+f(x_0,s_0,u^*) \cdot (\nabla\beta_\epsilon (x_0-y_0) - \delta \nabla_x \theta(y_0,x_0)) - L(x_0,s_0,u^*) \ge 0.
 \end{equation}
By the definition of $V$ \eqref{HJBt3}, assuming $\tau_0 \in [t_n,t_{n+1})$ we have
\begin{equation}\label{method}
V(y_0,\tau_0)- (t_{n+1}-\tau_0) \, L(y_0,\tau_0,u^*)+ -e^{-\lambda(t_{n+1}-\tau_0)} V(y_0+ (t_{n+1}-\tau_0) \,f(y_0,\tau_0,u^*),t_{n+1}) \le 0.
\end{equation}
Let us introduce
\begin{equation}
\xi(y,t)=V(y_0+y,\tau_0+t)-V(y_0,\tau_0)+ (\nabla \beta_\epsilon(x_0-y_0)+ \delta \nabla_y \theta (y_0,x_0)) \cdot y + t \nabla \eta_\alpha(\tau_0-s_0)
\end{equation}
and follows that
\begin{align*}
&\xi(y+z,t+\tau)-2\xi(y,t)+\xi(y-z,t-\tau)=\\
&\qquad\quad=V(y_0+y+z,\tau_0+t+\tau)-2V(y_0+y,\tau_0+t)+ V(y_0+y-z,\tau_0+t-\tau).
\end{align*}
Proposition \ref{semicon} ensures that the function $V$ is semiconcave. Thus, it follows that also the function $\xi$ with $\xi(0,0)=0$ is semiconcave.
Let us now check the last hypothesis of Lemma \ref{lemma1}. Since $(y_0,x_0,\tau_0,s_0)$ is a maximum point for $\zeta$, we obtain
\begin{align*}
&V(y_0+y,\tau_0+t)-V(y_0,\tau_0) \le \\
&\le\beta_\epsilon(y_0-x_0) -\beta_\epsilon(y_0+y-x_0)+ \eta_\alpha(\tau_0-s_0)- \eta_\alpha(\tau_0+t-s_0) +\delta [\theta(y_0,x_0)- \theta(y_0+y,x_0)],
\end{align*}
\begin{align*}
&\xi(y,t) \le \beta_\epsilon(y_0-x_0) -\beta_\epsilon(y_0+y-x_0) +\nabla \beta_\epsilon(x_0-y_0) \cdot y +\eta_\alpha(\tau_0-s_0)\\
&- \eta_\alpha(\tau_0+t-s_0)+ t \nabla\eta_\alpha(\tau_0-s_0)+ \delta (\theta(y_0,x_0)- \theta(y_0+y,x_0)+\nabla_y \theta(y_0,x_0) \cdot y ).
\end{align*}
We note that 
$$
\limsup_{(t,y)\rightarrow (0,0)} \frac{\xi(y,t)}{|y|+|t|} \le 0.
$$
Applying Lemma \ref{lemma1} with $y=(t_{n+1}-\tau_0)\, f(y_0,\tau_0,u^*)$ and $t=t_{n+1}-\tau_0$, we obtain
\begin{equation*}
V(y_0+(t_{n+1}-\tau_0)\, f(y_0,\tau_0,u^*),t_{n+1}) \le  V(y_0,\tau_0) -(t_{n+1}-\tau_0) (\nabla \beta_\epsilon (x_0-y_0)+ \delta \nabla_y \theta(y_0,x_0)) \cdot
\end{equation*}
\begin{equation}
 f(y_0,\tau_0,u^*)- (t_{n+1}-\tau_0) \nabla \eta_\alpha(\tau_0-s_0) + C_{\xi}(t_{n+1}-\tau_0)^2 (1+ |f(y_0,\tau_0,u^*)|^2).
\label{1}
\end{equation}
Inserting \eqref{1} in \eqref{method} and dividing by $t_{n+1}-\tau_0$ we obtain
\begin{align*}
&\frac{1-e^{-\lambda (t_{n+1}-\tau_0)}}{t_{n+1}-\tau_0} V(y_0,\tau_0) \le L(y_0,\tau_0,u^*)-e^{-\lambda (t_{n+1}-\tau_0)} (\nabla \beta_\epsilon(x_0-y_0)+\\
& \delta \nabla_y\theta(y_0,x_0) ) \cdot f(y_0,\tau_0,u^*)+\nabla \eta_\alpha(\tau_0-s_0) - C(t_{n+1}-\tau_0) (1+ |f(y_0,\tau_0,u^*)|^2)) .
\end{align*}
Finally, subtracting \eqref{2}, we obtain 
\begin{align*}
&\frac{1-e^{-\lambda (t_{n+1}-\tau_0)}}{t_{n+1}-\tau_0} V(y_0,\tau_0)- \lambda v(x_0,s_0) \le L(y_0,\tau_0,u^*)-L(x_0,s_0,u^*) + \\
& \nabla \beta_\epsilon(y_0-x_0) \cdot ( -e^{-\lambda (t_{n+1}-\tau_0)} f(y_0,\tau_0,u^*)+ f(x_0,s_0,u^*))+\\
&\nabla n_\alpha (\tau_0-s_0)(1-e^{-\lambda (t_{n+1}-\tau_0)})  + \delta\left( -e^{-\lambda (t_{n+1}-\tau_0)} \nabla_y \theta (y_0,x_0) \cdot f(y_0,\tau_0,u^*)\right)\\
&+\delta\left(- \nabla_x \theta (y_0,x_0) \cdot f(x_0,s_0,u^*)\right) + C (t_{n+1}-\tau_0)(1+ |f(y_0,\tau_0,u^*)|^2)\\
& \le  L_L(|y_0-x_0|+ |\tau_0-s_0|) + 2 (L_v + \delta) L_f(|y_0-x_0|+ |\tau_0-s_0|) +2L_v \Delta t +\\
& M \delta+C\Delta t (1+ M^2_f )  \le L(\epsilon^2+\alpha^2)+C\Delta t+M \delta.
\end{align*}
Since $\delta$ is arbitrary, choosing $\alpha=\epsilon=\sqrt{\Delta t}$, we obtain the assertion. $\Box$
\subsection*{Error estimate for a discrete control space}
In the previous results we did not take into account the error in the minimization procedure. To make our algorithm computationally feasible, we replace the continuous set of controls $U$ by a discrete set $U^{\Delta u}$ in order to compute the minimum by comparison over $U^{\Delta u}$.
Let us define the value function computed with a discrete set of controls as
\begin{equation}
\overline{V}(x,t)= \inf_{u \in \overline{\mathcal{U}}^\Delta} J_{x,t}^{\Delta t}(u),
\end{equation}
where
$$
\overline{\mathcal{U}}^\Delta=\{ u:[t,T) \rightarrow U^{\Delta u}, \mbox{ such that } u(s)= \sum_{k=n}^{\overline{N}-1} \alpha_k \chi_{[t_k,t_{k+1})}(s) \}.
$$
 We can obtain an error estimate for the comparison error adding the hypothesis of Lipschitz-continuity for $f$ and $L$ with respect to the state and the control uniformly in time, i.e.

\begin{align}
\begin{aligned}\label{Lf2}
&|f(x,u,s)-f(y,\tilde{u},s)| \le L_f (|x-y| + |u-\tilde{u}|),
\cr
 &|L(x,u,s)-L(y,\tilde{u},s)| \le L_L (|x-y| + |u-\tilde{u}|),\cr
&\quad\qquad\forall \, x,y \in \mathbb{R}^d, u,\tilde{u} \in U \subset \mathbb{R}^m, s \in [t,T].
\end{aligned}
\end{align}

\begin{proposition}
Under the assumptions \eqref{Lf2} and \eqref{Lg}, then
\begin{equation}
\sup_{(x,t) \in \mathbb{R}^d \times [0,T]}|V(x,t)-\overline{V}(x,t)| \le C(T,m) \Delta u,
\end{equation}
where $m$ is the dimension of the control set $U$. 
\end{proposition}

{\em Proof.}
First, we can observe that $\overline{\mathcal{U}}^\Delta \subset \mathcal{U}^\Delta$, then $V(x,t)\le \overline{V}(x,t)$.
Then, supposing $t \in [t_n, t_{n+1})$ and imposing $\lambda=0$ for simplicity, we obtain
$$
\overline{V}(x,t)-V(x,t) \le  \overline{V}^{n+1}(x + (t_{n+1}-t) f(x,\overline{u}^n,t))  + (t_{n+1}-t) L(x,\overline{u}^n,t)
$$
$$
- V^{n+1}(x + (t_{n+1}-t) f(x,u_*^n,t))  - (t_{n+1}-t) L(x,u_*^n,t),
$$
where
$$
u_*^n= \argmin_{u \in U} \{ V^{n+1}(x + (t_{n+1}-t) f(x,u,t)  +(t_{n+1}-t) L(x,u,t) \}
$$
and $\overline{u}^n$ is chosen such that $|\overline{u}^n-u_*^n| \le \frac{\sqrt{m}}{2} \Delta u$. This choice is possible since $U^{\Delta u}$ is a discretization of $U$ with step-size $\Delta u$ in all directions.
Then, if we carry on as Proposition $2.1$ in \cite{AFS18}, we will obtain
\begin{equation}
\overline{V}(x,t)-V(x,t) \le  L_L \sum_{k=n}^N \alpha_k \left( |\overline{x}^k - x_*^k| +|\overline{u}^k-u_*^k| \right) +L_g |\overline{x}^N- x_*^N|,
\label{VV}
\end{equation}
where
$$
\overline{x}^k:= x+ \sum_{j=n}^{k-1} \alpha_j f(\overline{x}^j, \overline{u}^j,\bar{t}_j), \quad
x_*^k=x +\sum_{j=n}^{k-1} \alpha_j f(x_*^j, u_*^j,\bar{t}_j),
$$
$$
\alpha_j= \begin{cases}
		 t_{n+1}-t & j=n\\
		 \Delta t & k \ge n+1
		 \end{cases} ,
		 \quad
\bar{t}_j= \begin{cases}
		 t & j=n\\
		 t_k & j \ge n+1
		 \end{cases} ,
$$
$$
u^{j}_*= \argmin_{u \in U} \{ V^{j+1}(x_*^j + \alpha_j f(x_*^j,u,\bar{t}_j))  + \alpha_j L(x_*^j,u,\bar{t}_j) \},\quad j \ge n,
$$
and $\overline{u}^{j}$ chosen such that
 $$
 |\overline{u}^{j}-u^{j}_*| \le \frac{\sqrt{m}}{2} \Delta u, \quad j \ge n.
 $$

By Gr\"onwall's lemma we obtain
\begin{equation}
|\overline{x}^k-x_*^k | \le e^{(t_k-t)L_f} (t_k-t) L_f \frac{\sqrt{m}}{2} \Delta u, \quad j \ge n,
\label{gron}
\end{equation}
and finally coupling \eqref{VV} and \eqref{gron} 
\begin{equation}
\sup_{(x,t) \in \R^d \times [0,T]} |\overline{V}(x,t)-V(x,t)| \le  \Delta u \frac{\sqrt{m}}{2} \left( e^{T L_f}T (L_L +L_g) +T L_L \right).\Box
\end{equation}

\begin{rmk}
If the function $f$ has a sub-linear growth, $i.e.$ there exist two constants $K_1$ and $K_2$ such that
$$
|f(x,u,s)| \le  K_1 |x| + K_2, \quad \forall\, x \in \mathbb{R}^d, u \in U \subset \mathbb{R}^m, s \in [t,T], 
$$
then the trajectory $y(s,u)$ lives in a compact set for all $s \in [t,T]$ and $u\in \mathcal{U}$ . Thus, the hypothesis of boundedness \eqref{Mf} can be avoided and the conditions \eqref{Lf}-\eqref{Cf} can be considered locally, i.e. holding on every compact subset.   
\end{rmk}

\subsection*{Error estimate for the TSA with pruning}
\label{sec:eep}
In the previous section we presented an error estimate for the TSA where a first order of convergence is achieved. However, as shown numerically in \cite{AFS18}, one can obtain the same order of convergence in the case of the pruned tree if the pruning tolerance $\ep$ in \eqref{tol_cri} is chosen properly. In this section, we extend the theoretical results of Section \ref{sec:eee} to the pruning case. Thus, let us define the {\em pruned trajectory:} 
\begin{equation}
\eta^{n+1}_{i_n}= \eta^n_{i_{n-1}}+ \Delta t f(\eta^n_{i_{n-1}}, u_{j_n},t_n) + \mathcal{E}_{\ep}(\eta^n_{i_{n-1}}+ \Delta t f(\eta^n_{i_{n-1}}, u_{j_n},t_n),\{\eta^{n+1}_{i}\}_{i}),
\label{perturbed}
\end{equation} 
where the indices $i_n$ and $j_n$ consider the pruning strategy with

\begin{equation}
\mathcal{E}_{\ep}(x,\{x_n\})= 
\begin{cases}
x_k-x & \mbox{ if } k\in\argmin\limits_n |x-x_n| \mbox{ and } |x-x_k|\le\ep, \\
0 & \mbox{otherwise.}
\end{cases}
\label{perturbation}
\end{equation}
The function $\mathcal{E}_{\ep}(x,\{x_n\})$ can be interpreted as a perturbation of the numerical scheme and $|\mathcal{E}_{\ep}(x,\{x_n\})| \le \ep$. As already done in \eqref{apy}, we consider the piecewise constant extension $\tilde{\eta}(s;u)$ of the approximation such that
\begin{equation}
\tilde{\eta}(s,u):={\eta}^{[s/ \Delta t]}(u)\quad s \in [t,T].
\end{equation}
First step is to prove that the tolerance must be chosen properly to guarantee a first order convergence of the scheme. The following result is obtained easily through Gr\"onwall's lemma. 
\begin{proposition}
Given the approximation $\tilde{y}(s;u,x,t)$ of equation \eqref{eq} and its perturbation $\tilde{\eta}(s;u,x,t)$ expressed in \eqref{perturbed}, then
\begin{equation}
|\tilde{y}(s;u,x,t)-\tilde{\eta}(s;u,x,t)| \le \ep  \frac{s-t}{\Delta t} e^{L_f(s-t)}, \quad
\forall s \in [t, T].
\label{error_pruning}
\end{equation}
Finally, to guarantee first order convergence, the tolerance must be chosen such that
\begin{equation}\label{tolerance}
\ep \le C \Delta t^2.
\end{equation}
\end{proposition}

Then we can define the {\em pruned} discrete cost functional as
\begin{align}
\begin{aligned}
J^{\Delta t,P}_{x,s}(u)&=(t_{n+1}-s) L(x, u, s)+ \Delta t \sum_{k=n+1}^{N-1} L(\eta^k, u, t_k) e^{-\lambda (t_k-s)} \\
&+ g(\eta^{\overline{N}})e^{-\lambda (t_N-s)},
\end{aligned}
\end{align}
for $s\in [t_n,t_{n+1})$  and define the {\em pruned} discrete value function as
$$
V^{P}(x,t):=\inf_{u \in \mathcal{U}^\Delta} J_{x,t}^{\Delta t,P}(u)
$$
which now satisfies the following equation 
\begin{align}
\begin{aligned}\label{HJBt4}
&V^P(x,s) = \min\limits_{u\in U} \left\{e^{-\lambda (t_{n+1}-s)} V^{P}\left(\eta_u^{n+1}(x),t_{n+1}\right) +(t_{n+1}-s)  L(x,u,s) \right\}, \\
&V^P(x,T) = g(x), \qquad x \in \mathbb{R}^d, s \in [t_n,t_{n+1}),
\end{aligned}
\end{align}
where $\eta_u^{n+1}(x) = x+(t_{n+1}-s) f(x,u,s)+ \mathcal{E}_{\ep}(x+(t_{n+1}-s) f(x,u,s),\{\eta^{n+1}_{i}\}_{i})$. Then, we can prove the following result, similar to Theorem \ref{unverso}.
\begin{proposition}
Under the hypothesis of Theorem \ref{unverso} and condition \eqref{tolerance}, we have
\begin{equation}
|V(x,t)-V^P(x,t)| \le C^*(T) \Delta t.
\label{value_pruned}
\end{equation}
\end{proposition}
%

Finally by triangular inequality and using estimate \eqref{eq:res} and \eqref{value_pruned}, we obtain the desired result:
\begin{equation}
|v(x,t)-V^P(x,t)| \le \left(C^*(T)+\widehat{C}(T)\right) \Delta t.
\end{equation}
whenever condition \eqref{tolerance} holds true.

\subsection{Pruning in the linear case}
In this section we provide an estimate on the cardinality of the pruned tree in the case of  linear dynamical systems:
\begin{equation}
\dot{y}(t)= Ay+ Bu,
\label{axbu}
\end{equation}
where the control $u$ is 1-dimensional for simplicity and $A\in\R^{d\times d}, B\in\R^{d}.$
Discretizing \eqref{axbu} using a one-step scheme, we obtain the following approximation
\begin{equation}
y^{n+1}=S_{\Delta} y^n+\Delta t \widetilde{S}_{\Delta} B u^n,
\label{disc_dyn}
\end{equation}
where $S_{\Delta}=I_n+\Delta t A$  and $\widetilde{S}_{\Delta}=I_n$ for an explicit Euler scheme, whereas $S_{\Delta}=\widetilde{S}_{\Delta}=(I_n-\Delta t A)^{-1}$ for Implicit Euler.  The matrix $I_n \in \mathbb{R}^{n \times n}$ is the identity matrix. 
 We will write $y^{n+1}_u$ with $u=(u^0,\ldots, u^n) \in (U^{\Delta u})^{n+1}$ to stress the dependence of the state $y^{n+1}$ on the discrete controls $\{u^i\}_{i=0}^n$. Here, we fix $U^{\Delta u}=[u_{min},u_{max}]$.
Let us state some results related to the application of the pruning criterion to this particular case. In what follows, we are going to fix the pruning tolerance $\epsilon_\mathcal{T}=C \Delta t^2$ which guarantees a first order convergence for the entire procedure.

\begin{proposition}
Let us consider two different discrete evolutions $y^{n}_u$ and $\widetilde{y}^{n}_{\widetilde{u}}$, where $u=(u^0,\ldots, u^{n-3}, u^{n-2},u^{n-1})$ and $\widetilde{u}=(u^0,\ldots, u^{n-3}, \widetilde{u}^{n-2}, \widetilde{u}^{n-1})$ such that 
\begin{equation}
u^{n-2}+ u^{n-1}=\widetilde{u}^{n-2}+ \widetilde{u}^{n-1}.
\label{sum_u}
\end{equation}
Then, $y^{n}_u$ and $\widetilde{y}^{n}_{\widetilde{u}}$ satisfy the pruning criterion \eqref{tol_cri} if 
\begin{equation}
\left\| (S_{\Delta}- I_n) \widetilde{S}_{\Delta} B \right\| |u^{n-1}-\widetilde{u}^{n-1}| \le C \Delta t .
\label{u12}
\end{equation}
Moreover, if \eqref{u12} holds for every pair $(u^{n-1}, \widetilde{u}^{n-1}) \in U^\Delta \times U^\Delta$, the pruning criterion is satisfied by every pair of nodes ($y^{n}_u$, $\widetilde{y}^{n}_{\widetilde{u}}$) such that
  \begin{equation}
  \sum_{i=0}^{n-1} u^i = \sum_{i=0}^{n-1} \widetilde{u}^i.
  \label{sum_pru}
  \end{equation}
Then, the levels of the tree grow at most linearly and the cardinality of the pruned tree $\mathcal{T}_P$ is bounded by the following estimate
$$
|\mathcal{T}_P| \le (M-1) \frac{\overline{N}(\overline{N}+1)}{2} +\overline{N}+1,
$$
or equivalently,
\begin{equation}
|\mathcal{T}_P| \le (M-1) \frac{(T-t)^2+(T-t) \Delta t}{2 \Delta t^2} +\frac{T-t}{\Delta t}+1,
\label{bound_card}
\end{equation}

where $\overline{N}=\frac{T-t}{\Delta t}$ is the number of time steps, $M$ is the number of discrete controls, $t$ is the initial time and $T$ is the final time.
\label{prop_quad}
\end{proposition}

\begin{proof}
By the definition of the discrete dynamics \eqref{disc_dyn} and equality \eqref{sum_u}, we obtain
$$
\| y^{n}_u-\widetilde{y}^{n}_{\widetilde{u}}\| \le \Delta t \left\| (S_{\Delta}- I_n) \widetilde{S}_{\Delta} B \right\| |u^{n-1}-\widetilde{u}^{n-1}|\le C\Delta t^2,
$$

where we have used condition \eqref{u12} in the last inequality. This proves that the pruning criterion \eqref{tol_cri} is fulfilled.\\
Let us pass to the second part of the proposition. 
Using assumption \eqref{sum_pru}, the term $u^0$ can be written as
 $$
 u^0=\sum_{i=0}^{n-1}\widetilde{u}^i-\sum_{i=1}^{n-1} u^i .
 $$
Then, applying the first part of the proposition, $u$ satisfies the pruning criterion with the vector
 $$
 \left(\widetilde{u}^0, \sum_{i=1}^{n-1}\widetilde{u}^i-\sum_{i=2}^{n-1} u^i , u^3, \ldots, u^{n-1} \right).
 $$
 The procedure can be iterated for consecutive components, 
 proving that $y^n_u$ and $y^n_{\widetilde{u}}$ satisfy the pruning criterion.
 We can pass to the final part of the proposition. First of all we note that $|\mathcal{T}^0_P|=1$ and $|\mathcal{T}^1_P|=M$ by construction.
For a general level $n$, we note that the cardinality of the pruned level $\mathcal{T}^n_P$ is less or equal to the cardinality of the set of vectors $(u^0,\ldots, u^{n-1})  \in [U^\Delta]^n$ with distinct sums. Indeed, if two vectors of controls $u,\widetilde{u} \in [U^\Delta]^n$ have the same sum, by the previous result we know that the corresponding discrete evolution $y^n_u$ and $\widetilde{y}^n_{\widetilde{u}}$ merge via the pruning criterion. The cardinalities of these two sets may be different if the pruning strategy cuts further nodes. 
 Hence, by combinatorial computations we obtain that $|\mathcal{T}^n_P| \le n(M-1)+1$ for $n \ge 0$.
  Finally, we observe
$$
|\mathcal{T}_P| =\sum_{i=0}^{\overline{N}}|\mathcal{T}^i_P |\le (M-1) \frac{\overline{N}(\overline{N}+1)}{2} +\overline{N}+1,
$$
 proving the result.

\end{proof}
\begin{rmk}
In the case of Explicit Euler, condition \eqref{u12} is equivalent to
$$
\Delta t \left\| A B \right\| |u^{n-1}-\widetilde{u
}^{n-1}| \le C \Delta t,$$
which is satisfied for every $\Delta t>0$ and every pair $(u^{n-1},\widetilde{u
}^{n-1}) \in U^\Delta \times U^\Delta$ fixing the constant 
$$
C=\left\| A B \right\| |u_{max}-u_{min}|.
$$

This constant $C$ may be very large, affecting the cardinality of the tree. Indeed, according to \eqref{bound_card}, the cardinality of the pruned tree grows as $1/\Delta t^2$. Hence, for an accuracy $\varepsilon$, the bound $C\Delta t<\varepsilon$ implies that the $|\mathcal{T}_P|$ grows as $C^2/\varepsilon^2$, showing the dependence of the cardinality on the constant $C$.

\end{rmk}

\section{Numerical Tests}\label{sec:test}
We now provide a numerical illustration of our theoretical results by means of two test cases. We first deal with the control of the heat equation and compare our approximate value function by a very accurate simulation of the Riccati equation. We, then, compute the order of convergence of our method. The second example deals with an advection equation with a bilinear control. Both examples consider a high dimensional problem.
 The numerical simulations reported in this paper are performed on a laptop with one core of an Intel Core i5-3, 1 GHz and 8GB RAM. The codes are written in Matlab. 

\subsection*{Test 1: Heat Equation}\label{sec:test42}

In the first test we deal with the control of the linear heat equation:
\begin{equation}
\begin{cases}
z_t=\sigma z_{xx}+ z_0(x) u(t), & (x,t) \in [0,1] \times [0,T], \\
 z(0,t)=z(1,t)=0, & t \in  [0,T], \\
 z(x,0)=z_0(x), & x \in [0,1],
\end{cases}
\label{heat}
\end{equation}
where the term $z_0(x)u(t)$ allows to introduce a spatial dependence to the control input. We set $T=1$, $\sigma=0.15$ and $z_0(x)=x-x^2.$ The discretization of \eqref{heat} with a centered finite difference method leads to a system of ODEs: 
\begin{equation}
\begin{cases}
\dot{y}(t)&=A y(t) + B u(t),\\
y(0) &=y_0
\label{ODE}
\end{cases}
\end{equation}
where $A\in\R^{d\times d}$ is the stiffness matrix and $B\in\R^d$ is given by  $B_i=z_0(x_i)$ for $i=1,\ldots,d$, and $x_i$ are the points of the spatial grid. The system of ODEs \eqref{ODE} is solved via an implicit Euler scheme. We refer to \cite{L07} for a complete description of finite difference methods for PDEs.

The cost functional we want to minimize reads:
\begin{equation*}
J_{y_0,t}(u)= \int_{t}^T \left( \|y(s)\|_2^2+  \frac{1}{100}|u(s)|^2 \right)\, ds +  \|y(T)\|_2^2.
\end{equation*}
For this problem we can derive an {\em exact} solution solving the well-known Riccati differential equation as in \cite{BBKS18,BIM12,S16}, if the control is unconstrained. We will compare the numerical value function computed by the TSA and by the Riccati equation. We will compute the following relative errors
$$
Err_{2}=\sqrt{ \frac{\sum_{n=0}^{\overline{N}} |V(y^n_*,t_n)-v(y^n_{R},t_n)|^2}{\sum_{n=0}^{\overline{N}} |v(y^n_{R},t_n)|^2}}, \quad
 Err_{\infty}=   \frac{ \max\limits_{n=0,...,{\overline{N}}} |V(y^n_*,t_n)-v(y^n_{R},t_n)|}{  \max\limits_{n=0,...,{\overline{N}}} |v(y^n_{R},t_n)|},
$$ 
where $\{y^n_*\}_n$ is the optimal trajectory computed via TSA, whereas $\{y^n_{R}\}_n$ is obtained solving the Riccati equation. In this example the control space is $U=[-1,0]$. This control space is derived by taking control bounds of the LQR problem. We recall that the control set in the LQR problem is unbounded whereas in the TSA it is bounded.
We set the dimension of \eqref{ODE}, $d=100$. We consider a time step equal to $\Delta t = 10^{-4}$ for the Riccati equation to obtain an accurate solution. To obtain an efficient pruning as explained in Remark \ref{efficient}, we construct a coarse tree using $\Delta t = 0.1$ and $2$ controls. Then, we build a snapshot matrix $Y$, where its columns correspond to the nodes of the coarse tree. We compute the Singular Value Decomposition of the matrix $Y$ and we consider the first left singular vector $\Psi$. In the construction of the tree, the nodes are first projected using the operator $\mathcal{P}:=\Psi^T$ in Remark \ref{efficient} and, then, compared to each other. Thus, the projected nodes that satisfy the pruning criterion are then compared in the original dimension. We recall that the vector $\Psi$ is the direction of maximum variance. This choice will reduce the numbers of comparisons in the full dimension. 

The pruning tolerance $\epsilon=C \Delta t^2$ limits the growth of tree to be at most quadratic with respect to the ${\overline{N}}$ time steps. The constant $C=2$ turns out to satisfy condition \eqref{u12} for every pair of controls in $U^{\Delta}$ and every choice of $\Delta t$ considered in the test. In the left panel of Figure \ref{fig3:cost11} we show that the cardinality of the pruned tree $\mathcal{T}_P$ for different time steps and 100 discrete controls. It is possible to note that the theoretical bound \eqref{bound_card} is satisfied and it turns out that the pruning strategy is cutting more nodes than predicted by Proposition \ref{prop_quad}.
A study of the order of convergence for the TSA with $100$ discrete controls is provided in Table \ref{test2:11contr}. The control space, in the LQR problem, is not discretized and that may provide some different results as the order of convergence. The order of convergence is around $1$ as expected with a pruned tree with tolerance $\ep = 2\Delta t^2$. The table also shows the number of nodes for each $\Delta t$ (second column) and how the cardinality of tree scales with respect to $\Delta t$ (third column). To give an idea of the importance of the pruning we mention that the ratio of the cardinality for a full tree between $\Delta t= 0.0125$ and $\Delta t= 0.025$ is of order $10^{80}$ whereas the order with a pruned tree is $34.67$. It is clear the huge difference and feasibility of our approach with a pruning criterion. This is due to the fact that the tree lives on a lower-dimensional manifold.

	\begin{table}[htbp]
		\centering
		\begin{tabular}{cccccccc}
					\toprule
					 $\Delta t$ & Nodes &  Nodes Ratio & CPU & $Err_{2}$ &  $Err_{\infty} $ & $Order_{2}$  & $Order_{\infty}$  \\
					\midrule
					0.1 &     94 & 	  &  \phantom{x}0.69s&    0.170 &   0.184 &   & \\
					0.05 &   442 &   \phantom{x}4.70	 &  \phantom{x} 3.03s &   7.8e-2 &  8.9e-2 &   1.12 &   1.04 \\
					0.025 &   3632 & \phantom{x}8.51 &  \phantom{x} 27s & 3.9e-2  & 4.1e-2  &  1.17 &    1.22 \\
					0.0125 &  130767 & \phantom{x}34.67 &  \phantom{x}1100s & 9.7e-3   &   1.3e-2 &  2.00 &  1.65  \\
			\bottomrule		
		\end{tabular}
   	\caption{Test 1: Error analysis and order of convergence for Implicit Euler scheme of the TSA with $\ep= 2\Delta t^2$ and 100 discrete controls.   }
	\label{test2:11contr}
	\end{table}

Finally, in the central panel of Figure \ref{fig3:cost11} we show the convergence of the cost functional decreasing the temporal step size $\Delta t$ for a given amount of controls ($100$ in the plot). As $\Delta t$ gets smaller, we better approximate our reference solution. In the right panel we show the optimal policy. Clearly, since our control space is not continuous, a chattering phenomenon appears, however it does not improve significantly the quality of our results.

\begin{figure}[htbp]
\centering
\includegraphics[scale=0.36]{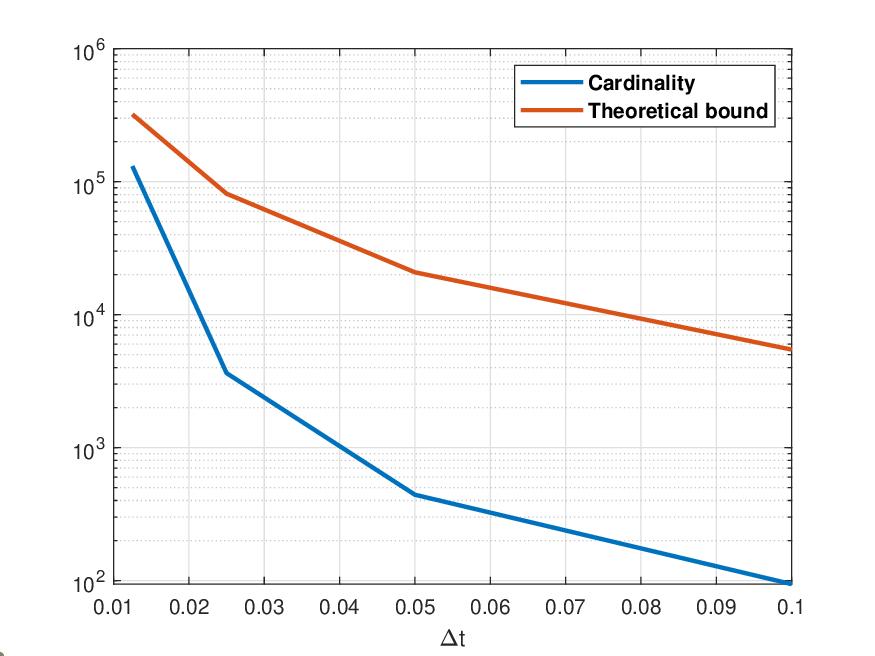}
       \includegraphics[scale=0.36]{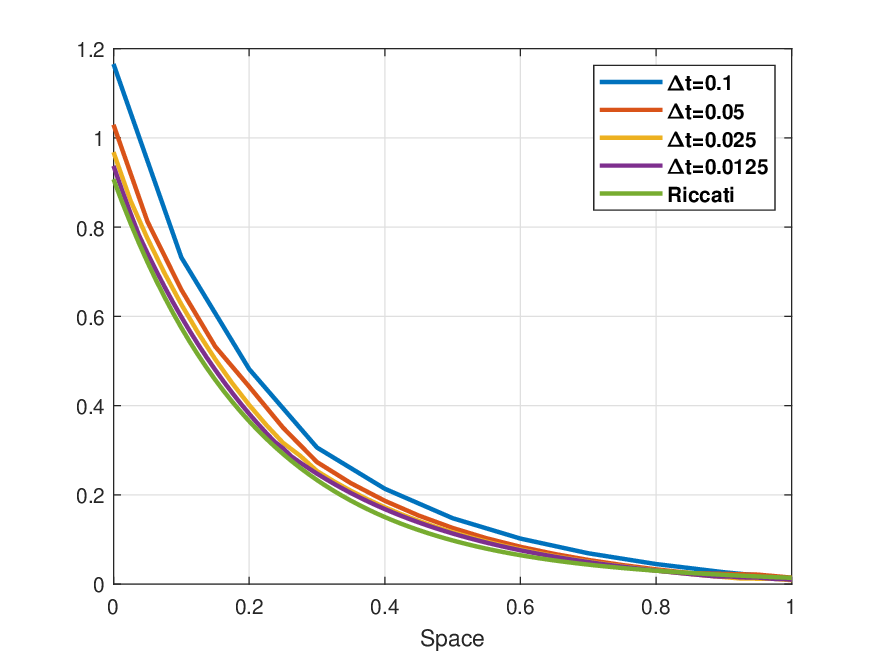}
       \includegraphics[scale=0.36]{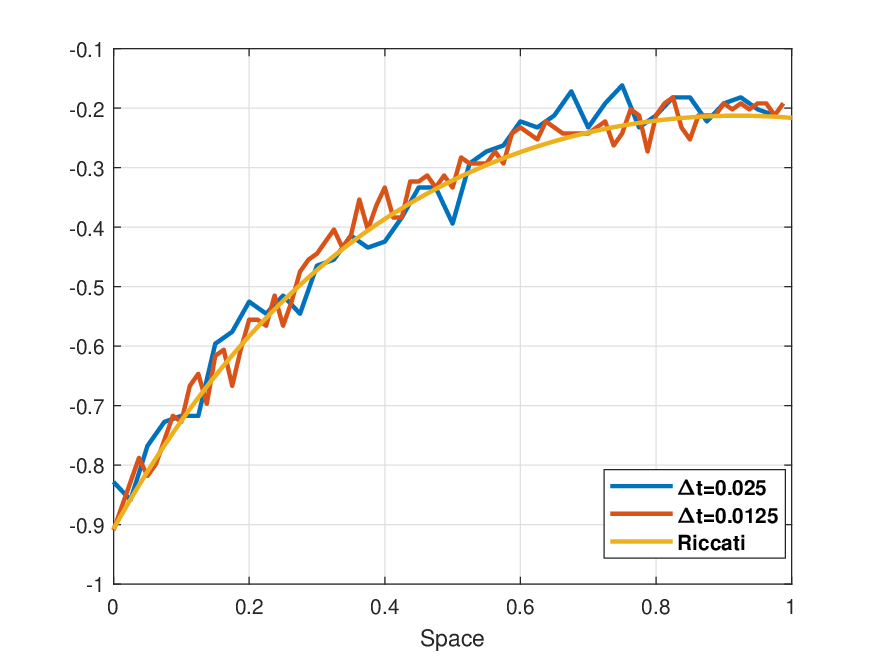} 
       \caption{Test 1: Cardinality of the trees compared with the theoretical estimate \eqref{bound_card} in logarithmic scale (left), cost functional (central) and optimal control (right) with $100$ discrete controls.}
       \label{fig3:cost11}
	\end{figure}

In Table \ref{test1:contr_var} we study the behaviour of the algorithm increasing the number of discrete controls. We observe that halving the control step the CPU time doubles and both errors decrease. We do not observe a first order convergence in the control discretization because for this problem a finer control space does not influence strongly the quality of our approximation.
	\begin{table}[htbp]
		\centering
		\begin{tabular}{cccccccc}
					\toprule
					 Discrete controls & Nodes & CPU & $Err_{2}$ &  $Err_{\infty} $ & $Order_{2}$  & $Order_{\infty}$  \\
					\midrule
4 & 2803 & 0.29s & 0.173 & 6.1e-2 &  &  \\
7 & 3101 & 0.55s & 0.138 & 5.2e-2 & 0.33 & 0.23 \\
13 & 3351 & 1.02s & 0.136 & 5.1e-2 & 0.02 & 0.03  \\
25 & 3370 & 2.03s & 0.100 & 4.3e-2 & 0.44 & 0.25 \\
49 & 3655 & 4.26s & 6.8e-2 & 3.6e-2 & 0.55 & 0.26  \\
			\bottomrule		
		\end{tabular}
   	\caption{Test 1: Error analysis and order of convergence for Implicit Euler scheme of the TSA with $\ep= 2\Delta t^2$ and $\Delta t=0.025$.   }
	\label{test1:contr_var}
	\end{table}

\subsection{Test 2: Advection equation}

The second test deals with the optimal control of the advection equation:
\begin{equation}
\begin{cases}
y_t+c y_{x}= y u(t) & (x,t) \in [0,3] \times [0,T], \\
 y(0,t)=y(3,t)=0 & t \in  [0,T], \\
 y(x,0)=y_0(x) & x \in [0,3],
\end{cases}
\label{transport}
\end{equation}
where $c=1.5$, $T=1$ and $y_0(x)=\sin(\pi x) \chi_{[0,1]}(x)$.  To discretize equation \eqref{transport}, we consider an upwind scheme in space, whereas we will apply an implicit Euler scheme in time (see \cite{L07}). We will minimize the following tracking cost functional:

\begin{equation*}
J_{y_0,t}(u)= \int_{t}^T  \int_0^3 \|y(x,s)-y_0(x-cs)\|^2 dx ds +  \int_0^3  \|y(x,T)-y_0(x-cT)\|^2 dx,
\end{equation*}

where the desired state is the solution of the uncontrolled problem, i.e.  $u\equiv 0$. The exact uncontrolled dynamics preserves over time the $L^\infty$ and $\L^2$ norm, while the upwind approximation introduces a term of numerical diffusion (see e.g. \cite{L07}). Specifically, in this case the control $u \in [0,1]$ will reduce the numerical diffusion of the scheme. To solve the control problem we use the TSA with $U=\{0,1\}$ and the following pruning criterion: $\ep=\Delta t^2$. We apply again the efficient pruning (see Remark \ref{efficient}), building the snapshots matrix with $\Delta t=0.1$ and $2$ controls.
In Table \ref{test2:2contr} we compute the error and order of convergence. In this case the exact value function is $v(t,x)=0$ obtained with the control $u\equiv 0$ in \eqref{transport}. As expected, we achieve first of order convergence.

\begin{table}[htbp]
		\centering
		\begin{tabular}{cccccccc}
					\toprule
					 $\Delta t$ & Nodes & CPU & $Err_{2}$ &  $Err_{\infty} $ & $Order_{2}$  & $Order_{\infty}$  \\
					\midrule
		0.05		 &  231 &	0.01s & 0.1 & 0.11 &  & \\
 0.025 &   861 & 0.05s & 0.05 & 0.054 & 0.99 &  1.00 \\
0.0125  &  3321 &  0.17s &0.025 &  0.027 & 0.99 & 0.99 \\
0.00625  &   13041 &  0.58s &0.014 & 0.0154 & 0.82 & 0.83\\
			\bottomrule		
		\end{tabular}
   	\caption{Test 2: Error analysis and order of convergence for Implicit Euler scheme of the TSA with $\ep= \Delta t^2$ and $2$ discrete controls.}
	\label{test2:2contr}
	\end{table}

It is possible to note that the increasing ratio for the nodes is almost equal to $4$, due to a linear growth of the cardinality. Indeed, in the left panel of Figure \ref{fig:transport}, we present the number of nodes for each time level of the tree. The linear growth in the time levels leads to a quadratic cardinality of the pruned tree e.g. $|\mathcal{T}^P|= \overline{N}(\overline{N}+1)/2$. In the middle panel of Figure \ref{fig:transport}, we show the final configuration for the uncontrolled and controlled solution. 
In Table \ref{test2:norms} we compare the $L^\infty$ norm and $L^2$ norm for the three cases studied. It is possible to note that both norms for the controlled solution are closer to the exact one than in the uncontrolled case since the control reduces the diffusivity of the scheme.
 In the right panel we show the optimal control computed by our method. 
We can note that the numerical optimal control oscillates between the two values to counteract the numerical diffusion.

Finally, in Table \ref{test2:5contr} we report the error analysis and the order of convergence with $5$ equidistant discrete controls. If we compare Table \ref{test2:5contr} and Table \ref{test2:2contr}, we can only note a slight improvement on the order of convergence. In the left panel of Figure \ref{fig:transport22} we report the growth of the cardinality fixing $\Delta t=0.00625$ and $5$ controls. We can observe again a linear growth, but in this case the growth rate is equal to $4$. In the right panel of Figure \ref{fig:transport22}, we show the optimal control which exhibits again an oscillating behaviour between two values.

\begin{figure}[htbp]
\centering
       \includegraphics[scale=0.3]{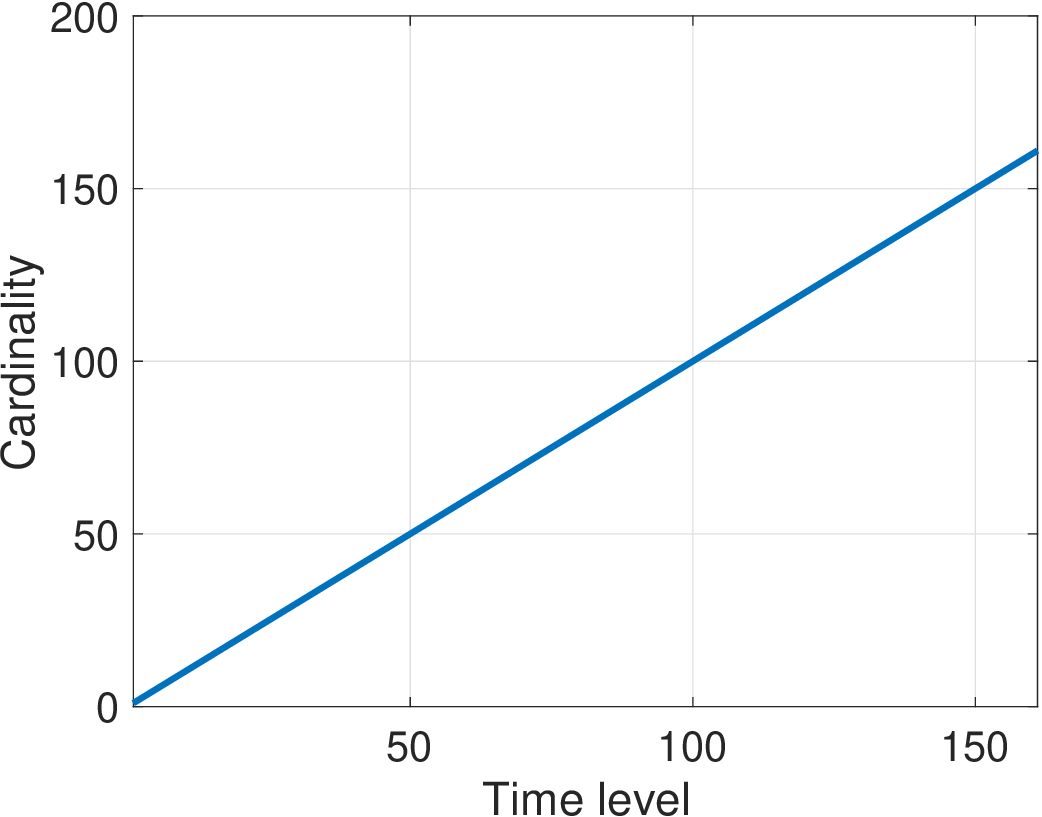}
       \includegraphics[scale=0.3]{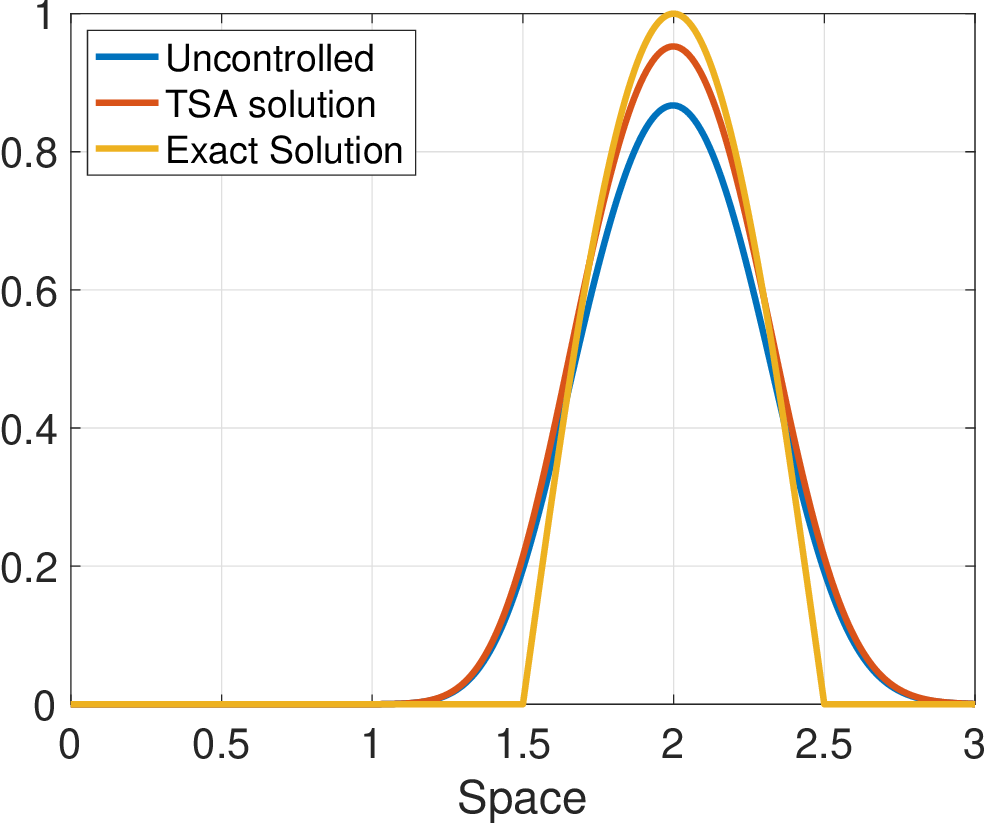} 
        \includegraphics[scale=0.3]{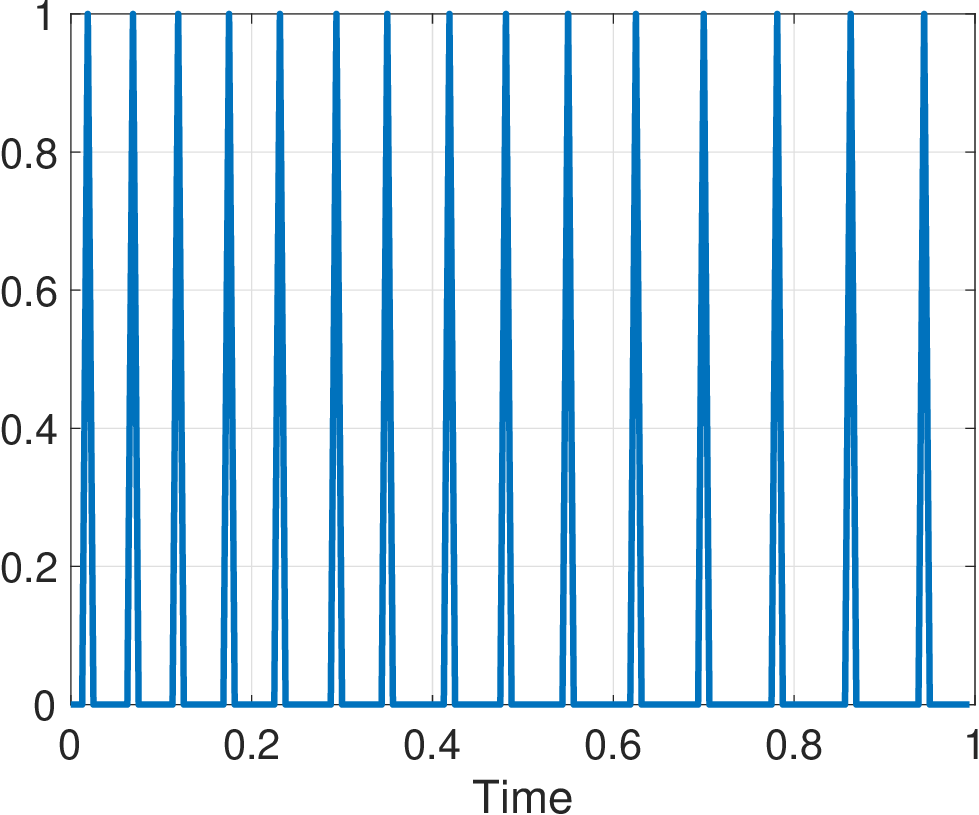} 
       \caption{Test 2: Number of nodes for each time level (left), solution at final time (middle) and optimal control (right) with $2$ discrete controls and $\Delta t=0.00625$ and $\ep=\Delta t^2$.}
       \label{fig:transport}
	\end{figure}
	
		\begin{table}[htbp]
		\centering
		\begin{tabular}{cccccccc}
					\toprule
					 & $\Vert y(T) \Vert_\infty$ & $\Vert y(T) \Vert_2$  \\
					\midrule
Exact & 1  & 0.707   \\
Uncontrolled &   0.867 & 0.636 \\
Controlled & 0.953 & 0.699  \\

			\bottomrule		
		\end{tabular}
   	\caption{Test 2: Comparison of the $L^\infty$ norm and $L^2$ norm for the exact, the uncontrolled and the controlled solutions at final time with $\Delta t=0.00625$.   }
	\label{test2:norms}
	\end{table}
	
	\begin{table}[htbp]
		\centering
		\begin{tabular}{cccccccc}
					\toprule
					 $\Delta t$ & Nodes & CPU & $Err_{2}$ &  $Err_{\infty} $ & $Order_{2}$  & $Order_{\infty}$  \\
					\midrule
 0.05 & 861 &  0.12s  &  0.1 &   0.11     &   & \\
 0.025  & 3321  &  0.39s  & 0.05  &  0.054   & 1   & 1 \\
     0.0125 & 3041 &  1.46s  &  0.025 &   0.027   & 0.99  &  0.99 \\
     0.00625 & 51681 &  5.7s &  0.0142 &   0.0154  &  0.83  &  0.83 \\
			\bottomrule		
		\end{tabular}
   	\caption{Test 2: Error analysis and order of convergence for Implicit Euler scheme of the TSA with $\ep= \Delta t^2$ and $5$ discrete controls.}
	\label{test2:5contr}
	\end{table}
	
	\begin{figure}[htbp]
\centering
       \includegraphics[scale=0.4]{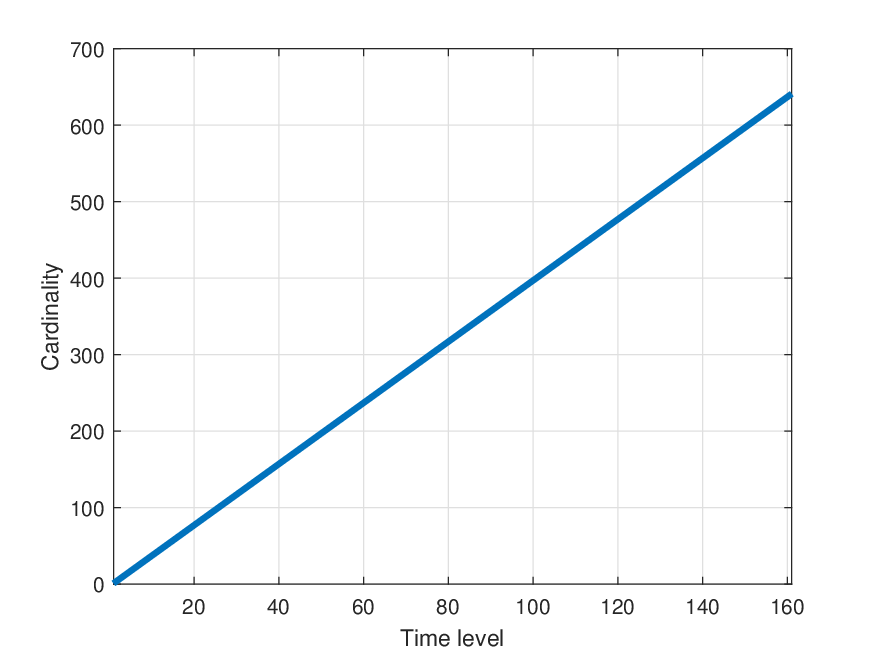}
        \includegraphics[scale=0.4]{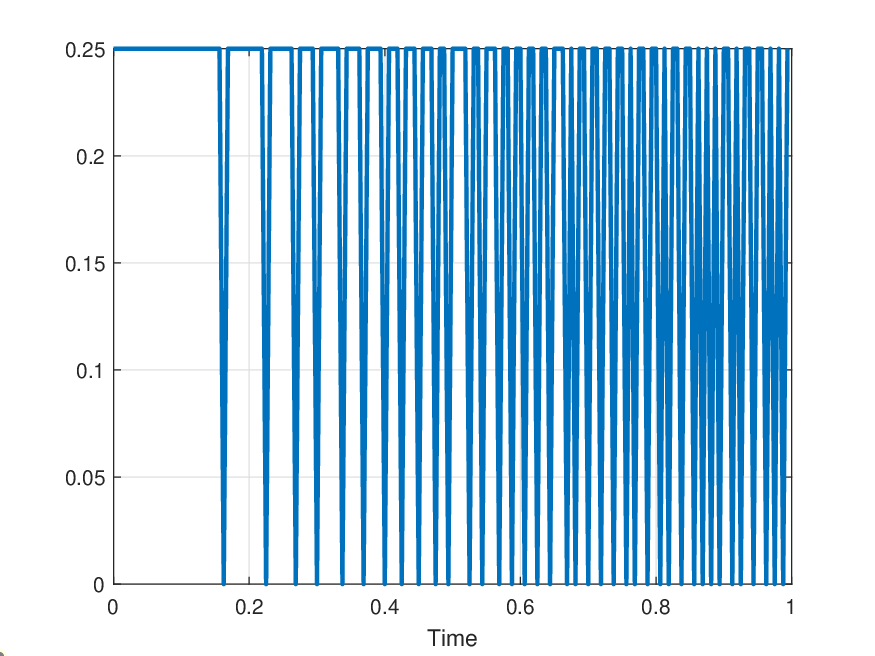} 
       \caption{Test 2: Number of nodes for each time level (left) and optimal control (right) with $5$ discrete controls and $\Delta t=0.00625$ and $\ep=\Delta t^2$.}
       \label{fig:transport22}
	\end{figure}
	
\section{Conclusion and future work}
\label{sec:con}
In this work we have proved error estimates for the TSA presented in \cite{AFS18}. The tree structure algorithm allows us to achieve the same order of convergence of the numerical method used in the time discretization of the dynamics. Our error estimate improves previous existing results on the convergence of the semi-discrete value function adding the semiconcavity assumption. Numerical tests presented in the last section and in \cite{AFS18} confirm the estimate and the relevance of semiconcavity in the approximation.

The cardinality of the tree increases as the number of the control does and the time step size $\Delta t$ decreases, so we need to prune the tree to reduce the complexity and to save in memory allocations and CPU time. The pruning technique is crucial to produce a more efficient algorithm. In particular, we have shown that if the pruning technique has a reasonable tolerance $\ep$, e.g. one order higher of the order of convergence of the numerical method of the ODE, we can achieve the same order of the TSA method without pruning.  

As future work we aim at analyzing the numerical methods for the synthesis of feedback controls based on the tree structure algorithm. 

\bibliographystyle{siamplain}

\end{document}